\documentclass[12pt,A4]{article}
\usepackage[french,english]{babel}
\usepackage{amsmath,amsfonts,amssymb,amsthm,mathrsfs,bbm}
\usepackage[font=sf, labelfont={sf,bf}, margin=1cm]{caption}
\usepackage{graphicx,graphics}
\usepackage{epsfig}
\usepackage{latexsym}
\usepackage[applemac]{inputenc}
\usepackage{ae,aecompl}
\usepackage{pstricks}
\usepackage{enumerate}
\usepackage{xcolor}

\usepackage[pdfpagemode=UseNone,bookmarksopen=false,colorlinks=true,urlcolor=blue,citecolor=blue,citebordercolor=blue,linkcolor=blue]{hyperref}

\usepackage[normalem]{ulem}
\usepackage[top=2.5cm,left=1.5cm,right=1.5cm,bottom=2.5cm]{geometry}
\usepackage{comment}
\usepackage{eulervm,palatino}

\DeclareGraphicsExtensions{.jpg,.pdf,.eps}

\newcommand{\op}[1]{\operatorname{#1 }}

\newcommand{\R}{\mathbb R}
\newcommand{\N}{\mathbb N}
\newcommand{\Z}{\mathbb Z}

\renewcommand{\P}[2]{\mathbb{P}_{#2}\left( #1 \right)}

\def\build#1_#2^#3{\mathrel{
\mathop{\kern 0pt#1}\limits_{#2}^{#3}}}

\newtheorem{theorem}{Theorem}

\newtheorem{proposition}[theorem]{Proposition}
\newtheorem{lemma}[theorem]{Lemma}

\renewcommand{\P}[1]{\mathbb{P}\left[#1\right]}

\theoremstyle{definition}

\newcommand\Es[1]{\mathbb{E}\left[#1\right]}
\renewcommand\Pr[1]{\mathbb{P}\left(#1\right)}
\newcommand\GWmu[1]{\GW_{ \mu}\left(#1\right)}
\newcommand\Esmu[1]{\GW_{ \mu}\left[#1\right]}

\def \N {\mathbb N}
\def \R {\mathbb R}

\def \P {\mathbb{P}}

\def \H {\mathsf{Height}}

\def \GW {\mathsf{GW}}

\def \t {\mathfrak{t}}

\def \ddiss {  \mathrm{d} _{ \mathcal{D}}}

\usepackage{enumerate}
\def\llbracket{[\hspace{-.10em} [ }
\def\rrbracket{ ] \hspace{-.10em}]}
\title{  \vspace {-2cm}\textbf{The  CRT is the scaling limit of random dissections}}
\date{}
\author{Nicolas Curien\thanks{Université Paris 6 and CNRS, E-mail: nicolas.curien@gmail.com},\quad  Bénédicte Haas\thanks{ Universit\'e Paris-Dauphine and \'Ecole normale supérieure, E-mail: haas@ceremade.dauphine.fr} \quad and \quad Igor Kortchemski\thanks{DMA, École Normale Supérieure, E-mail: igor.kortchemski@normalesup.org}}

\DeclareSymbolFont{extraup}{U}{zavm}{m}{n}
\DeclareMathSymbol{\varheart}{\mathalpha}{extraup}{86}
\DeclareMathSymbol{\vardiamond}{\mathalpha}{extraup}{87}

\makeatletter
\renewcommand*{\@fnsymbol}[1]{\ensuremath{\ifcase#1\or  \vardiamond \or \clubsuit\or \spadesuit\or
   \mathsection\or \mathparagraph\or \|\or **\or \dagger\dagger
   \or \ddagger\ddagger \else\@ctrerr\fi}}
\makeatother

\begin{document}
\maketitle

\let\thefootnote\relax\footnotetext{ \\\emph{MSC2010 subject classiﬁcations}. Primary 60J80,05C80 ; secondary 05C05. \\
 \emph{Keywords and phrases.} Random dissections, Galton--Watson trees, scaling limits, Brownian Continuum Random Tree,  Gromov--Hausdorff topology}
 
\vspace {-0.5cm}

\begin{abstract} We study the graph structure of large random dissections of polygons sampled according to Boltzmann weights, which encompasses the case of uniform dissections or uniform $p$-angulations. As their number of vertices $n$ goes to infinity, we show that these random graphs, rescaled by $n^{-1/2}$,  converge in the Gromov--Hausdorff sense towards a multiple of  Aldous' Brownian  tree when the weights decrease sufficiently fast. The scaling constant depends on the Boltzmann weights in a rather amusing and intriguing way, and is computed by making use of a Markov chain which compares the length of geodesics in dissections with the length of geodesics in their dual trees. 
\end{abstract}

 \begin{figure}[!h]
 \begin{center}
  \includegraphics[width=0.95  \linewidth]{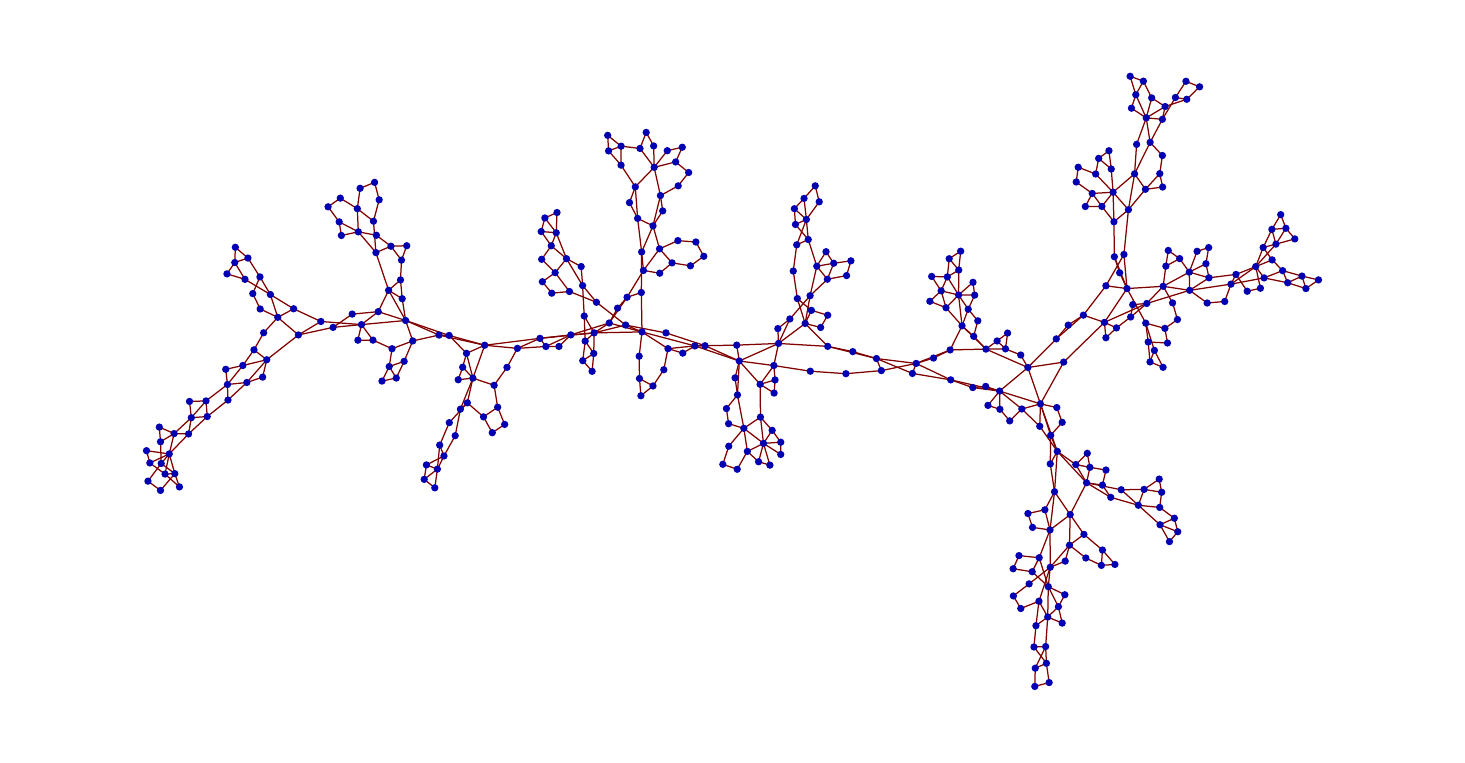}
 \caption{ \label{fig:largediss} A uniform dissection of a polygon with 387 vertices, embedded non isometrically in the plane.}
 \end{center}
 \end{figure}

 \vfill
 \pagebreak
\section{Introduction}
Let $ \mathcal{P}_{n}$ be the convex polygon inscribed in the unit disk $\overline{\mathbb{D}}$ of 
the complex plane whose vertices are the $n$-th roots of unity. A \emph{dissection} of $\mathcal{P}_{n}$  is by definition the union of the sides of $\mathcal{P}_{n}$ together with a collection of diagonals that may intersect only at their endpoints. A \emph{triangulation} (resp. a \emph{$p$-angulation} for $p \geq 3$) is a dissection
whose inner faces are all triangles (resp.\,$p$-gons). 

\begin{figure}[!h]
 \begin{center}
 \includegraphics[height=3cm]{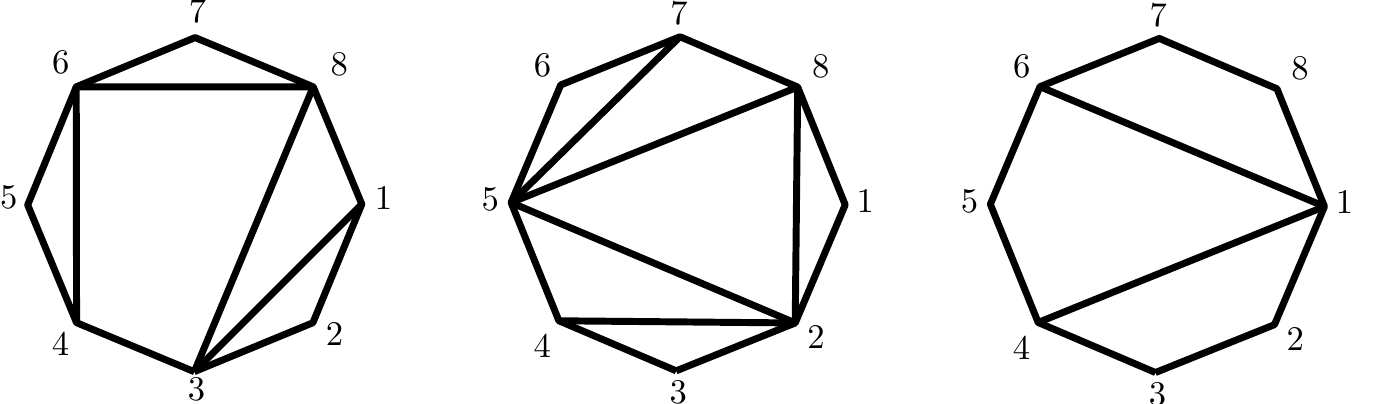}
 \caption{ \label{fig1} A dissection, a triangulation and a quadrangulation of the octogon.}
 \end{center}
 \end{figure}
 In \cite{Ald94b}, Aldous studied random uniform triangulations of $ \mathcal{P}_{n}$ seen as closed subsets of $ \overline{\mathbb{D}}$ 
 (see Fig.\,\ref{fig1}), and proved convergence, as $n \to \infty$, towards a random closed  subset   of $ \overline{\mathbb{D}}$ of Hausdorff dimension ${3}/{2}$ called the ``Brownian triangulation''. This approach has been pursued in \cite{CKdissections} in the case of uniform dissections, see also \cite{CLGrecursive,Kor11} for related models. In this work, instead of viewing dissections as subsets of the unit disk, we view them as compact metric spaces by equipping the vertices of the polygon with the graph distance (every edge has unit length).

Graph properties (such as  maximal vertex or face degrees, diameter, etc.) of large random dissections have attracted a lot of attention in the combinatorial literature. In particular, it has been noted that the combinatorial structure of dissections (and more generally of non-crossing configurations) is very close to that of plane trees  (see Fig. \ref{fig:largediss} for an illustration). 
For instance, the number of dissections of $\mathcal{P}_{n}$ exhibits the $n^{-3/2}$ polynomial correction \cite{FN99}, characteristic in the counting of trees. Also, various models of random dissections of $\mathcal{P}_{n}$  have maximal {vertex or face} degree{s} of order $ \log(n)$ \cite{BPS08,CKdissections,DFHN99,GW00} and diameter of order $ \sqrt{n}$ \cite{DMN12}, thus suggesting a ``tree-like'' structure.

In this work, we show that many different models of large random dissections, suitably rescaled, converge towards the Brownian Continuum Random Tree (CRT) introduced by Aldous in \cite{Ald91a}.
 The latter convergence holds in distribution with respect to the Gromov--Hausdorff topology which gives sense to convergence 
of compact metric spaces, see Section \ref{sec:GH} {for background}. 
 \bigskip

\noindent \textbf{Boltzmann dissections.} \quad We will work with the model of random Boltzmann dissections introduced in  \cite{Kor11}. Let $ \mu=(\mu_j)_{j \geq 0}$ be a probability distribution 
on the nonnegative integers $\mathbb Z_+=\{0,1, \ldots  \}$ such that $\mu_1=0$ and the mean of $  \mu$ is equal to $1$ ($\mu$ is said to be critical). For every integer $n\geq  3$ for which it makes sense, the  Boltzmann probability measure $\mathbb{P} ^ {\mu}_ {n}$ is the probability measure on the set of all dissections of $ \mathcal{P}_{n}$ defined by
$$ \mathbb{P} ^ {\mu}_ {n}(\omega)=Z_{n}^{-1} \prod_{f \textrm{ inner face of } \omega}
\mu_{\deg(f)-1}, $$ where $\deg(f)$ is the
degree of the face $f$, that is the number of edges in the boundary of $f$, and $Z_n$ is a normalizing constant. Note that the definition of $\mathbb{P} ^ { \mu}_ {n}$ only involves  $\mu_2, \mu_3, \ldots$, the initial weights $ \mu_0$ and  $\mu_{1}$ being here in order that $ \mu$ defines a {critical} probability measure. This will later be useful, see Proposition \ref{prop:GW}. We also point out that the hypothesis $\sum_{i \geq 2} i \mu_i=1$ is not as restrictive as it may first  appear, this is discussed in the remark before Section 2.2. In the following, all the statements have to be implicitly restricted to the values of $n$ for which the definition of $ \mathbb{P}_{n}^\mu$ makes sense.

Throughout the paper, 
$  \mathcal{D}^\mu_{n} $ denotes
a random dissection of $ \mathcal{P}_{n}$ distributed according to $ \mathbb{P} ^ {\mu}_ {n}$, which is  endowed  with the graph distance. More generally, it is implicit in this paper that all graphs are equipped with the graph distance.
We use the version of the CRT which is constructed from a normalized Brownian excursion $ \mathbf{e}$, see \cite[Section 2]{LG05}, and we will denote it by $ \mathcal{T}_{ \mathbf{e}}$. If $ \mathcal{M}$ is a metric space, the notation $ c \cdot  \mathcal{M}$ stands for the metric space obtained from $ \mathcal{M}$ by multiplying all distances by $c>0$. We are now ready to state our main result. 

 \medskip 

\begin{theorem} \label{thm:main} Let $\mu$ be a probability measure on $\{0,2,3,\ldots\}$ of mean $1$ and assume that  $ \sum_{i \geq 0} e^ { \lambda  i} \mu_{i}< \infty$ for some $\lambda>0$.  Set $\mu_{0}+ \mu_{2} + \mu_{4} + \cdots = \mu_{2 \mathbb{\mathbb Z_+}}$ and let $ \sigma^2 \in (0, \infty)$ be the variance of $ \mu$. Finally set $ c( \mu)=c_{ \mathrm{tree}}(\mu) \cdot c_{ \mathrm{\mathrm{geo}}}(\mu)$, where
$$ c_{ \mathrm{tree}}(\mu) := \frac{2}{ \sigma \sqrt{\mu_{0}}}, \qquad  c_{ \mathrm{\mathrm{geo}}}(\mu) := \frac{1}{4}\left( \sigma^2+ \frac{ \mu_0 \mu_{2 \mathbb{Z_+}}}{2\mu_{2 \mathbb{Z_+}}- \mu_0} \right).$$
Then the following convergence holds in distribution for the Gromov--Hausdorff topology
  \begin{eqnarray} \frac{1}{ \sqrt{n}} \cdot \mathcal{D}^\mu_{n} & \quad \xrightarrow[n\to\infty]{(d)} \quad &     c( \mu) \cdot  \mathcal{T}_{ \mathbf{e}}.  \label{eq:thm}\end{eqnarray}  \end{theorem}

The reason why the constant $c( \mu)$ is split into two parts is explained below. 
 \bigskip

\noindent \textbf{Examples.} \quad  Let us give a  few important special cases (see  Section \ref{examples} for other examples).
\begin{itemize}
\item { \bf {Uniform} $p$-angulations.}
Consider an integer $ p \geq 3$. If $ \mu^{(p)}_0 = 1-1/(p-1)$, $ \mu^{(p)}_ {p-1}=1/(p-1)$ and $ \mu^{(p)}_i=0$ otherwise, then $\mathbb{P} ^ {\mu^{(p)}}_ {n}$ is the uniform measure over all $p$-angulations of  $ \mathcal{P}_{n}$ (in that case, we must restrict our attention to values of $n$ such that $n-2$ is a multiple of $p-2$
). We thus get
$$ c(  \mu^{(p)})=\frac{p}{2 \sqrt{p-1}} \ \  \mbox{ for $p$ even \quad ($p \geq 4$)}  \quad\mbox{ and}   \quad c(  \mu^{(p)})=\frac{(p+1)\sqrt{p-1}}{2p} \ \ \mbox{ for $p$ odd \quad ($p \geq 3$)}.$$
It is interesting to note that $c(  \mu^{(p)})$ is increasing in $p$.
\item{ \bf Uniform dissections.}  If $ \mu_0 = 2 - \sqrt {2}, \mu_{1}=0$ and $ \mu_i= ((2 - \sqrt {2})/2) ^ {i-1}$ for every $i \geq 2$, then $ \mathbb{P} ^ {\mu}_ {n}$ is the uniform measure on the set of all dissections of $  \mathcal{P}_{n}$ (see \cite[Proposition 2.3]{CKdissections}). In this case, 
$$ c( \mu)= \frac{1}{7}(3 + \sqrt{2})2^{3/4}  \quad \simeq  \quad 1.0605.$$
\end{itemize}

If $\mu$ is critical but has a heavy tail, i.e. $\mu_{k} \sim c \cdot k^{-(1+\alpha)} $ as $ k \rightarrow \infty$ for fixed $\alpha \in (1,2)$ and $c>0$, a drastically different behavior occurs. Indeed, in the recent work \cite{CK13}, it is shown that the random metric space $ \mathcal{D}_{n}^{\mu}$, now renormalized by $n^{1/\alpha}$, converges towards the stable looptree of parameter $\alpha$ which is also introduced in \cite{CK13}.

 \bigskip

\noindent \textbf{Combinatorial applications.} \quad Theorem \ref{thm:main} implies that $ \Es{F(\mathcal{D}^\mu_{n}/\sqrt{n})} \rightarrow \Es{F(c(\mu)\cdot\mathcal{T}_{ \mathbf{e}})}$ as $ n \rightarrow \infty$ for every bounded continuous function $F$ (defined on the set of compact metric spaces) with respect to the Gromov--Hausdorff topology. By controlling the speed of convergence in Theorem \ref {thm:main}, we will actually show that the last convergence holds more generally for functions  $F$  such that  $F( \mathcal{M}) \leq  C \cdot \mathsf{Diam}(\mathcal{M})^p$ for every compact metric space $  \mathcal{M}$ and fixed $C,p > 0$, where $ \mathsf{Diam}( \cdot)$ stands for the diameter, which is by definition the maximal distance between two points in a compact metric space.  

As a consequence, we obtain the asymptotic behavior of all positive moments of different statistics of $\mathcal{D}^\mu_{n}$, such as the diameter, the radius or the height of a random vertex, see Section \ref{sec:moments}. For instance,  in the case of uniform dissections, we get 
 \begin{eqnarray*} \mathbb{E}\Big[\mathsf{Diam}( \mathcal{D}_{n}^\mu)\Big] &  \displaystyle \quad\mathop{ \sim}_{n \rightarrow \infty} \quad  &
  \frac{1}{21}(3 + \sqrt{2}) 2^{9/4} \ \sqrt{ \pi n} \quad \simeq \quad 1.7723 \ \sqrt{n}.\end{eqnarray*}
This strengthens a result of \cite[Section 5]{DMN12}.

 \bigskip

\noindent \textbf{Strategy of the proof  and organization of the paper.} \quad We have deliberately split the scaling constant appearing in \eqref{eq:thm} into two parts in order to reflect the two main steps of the proof. 
 
 First, in Section \ref{sec:duality},  we associate with every dissection $\mathcal{D}_{n}^{\mu}$ a ``dual'' tree denoted by $ \phi( \mathcal{D}_{n}^{\mu})$ (see Figure \ref{fig:dual}). It turns out that $ \phi( \mathcal{D}_{n}^{\mu})$ is a Galton--Watson tree with offspring distribution $\mu$ and conditioned on having  $n-1$ leaves (Proposition \ref{prop:GW}). Since the work of Aldous, it is well known that, under a finite variance condition, Galton--Watson trees conditioned on having $n$ vertices, and scaled by  $ \sqrt{n}$, converge towards the Brownian CRT. Here, the conditioning is different and involves the number of leaves. However, such a situation was studied in \cite{Kor12,Riz11} and it follows that $\phi( \mathcal{D}_{n}^{\mu})/\sqrt{n}$ converges in distribution towards $c_{ \mathrm{tree}}(\mu)   \cdot  \mathcal{T}_{ \mathbf{e}}$.

The second step consists in showing that the random metric spaces $\mathcal{D}_{n}^{\mu}$ and $\phi( \mathcal{D}_{n}^{\mu})$ are roughly proportional to each other, the proportionality constant being
precisely $ c_{ \mathrm{\mathrm{geo}}}(\mu)$.
To this end, we show that the length of a geodesic in $ \mathcal{D}_{n}^{\mu}$ starting from the root and targeting a typical vertex  is described by an exploration algorithm indexed by the associated geodesic in the tree $ \phi( \mathcal{D}_{n}^{\mu})$. See Section \ref{sec:geoexplo} for precise statements. In order to obtain some information on the asymptotic behavior of this exploration procedure, we first study in Section \ref{sec:makov} the case of the critical Galton--Watson tree conditioned to survive where the geodesic exploration yields a Markov chain. For each step along the geodesic in the tree, the mean increment (with respect to the stationary distribution of the Markov chain) along the geodesic in the dissection is precisely $ c_{ \mathrm{\mathrm{geo}}}( \mu)$. 
In Section \ref{sec:proof}, we then control all the distances in $ \phi( \mathcal{D}_\mu^n)$  by using large deviations for the Markov chain. This allows us to estimate the Gromov--Hausdorff distance between $\mathcal{D}_{n}^{\mu}$ and $\phi( \mathcal{D}_{n}^{\mu})$ (Proposition \ref {prop:utile}) and  yields Theorem \ref{thm:main}.

Last, we develop  in Section \ref{sec:appl} applications and extensions of Theorem \ref{thm:main}. In particular, we study the asymptotics of positive moments of several statistics of $\mathcal D_n^{\mu}$ and set up a result similar to Theorem \ref{thm:main} for the scaling limits of  discrete looptrees associated to large Galton--Watson trees.
 
\bigskip 
Let us also mention that in \cite{AlMar08}, Albenque and Marckert proved a result similar to Theorem \ref{thm:main} for the uniform stack triangulations. Their approach also relies on a comparison of the distances in the graphs and in some dual trees. See also \cite{JanSte14, Bet14} for other examples of random maps that are not trees and that converge towards the Brownian CRT.  

 \bigskip

\noindent \textbf{Acknowledgments.} \quad We are indebted to Marc Noy for stimulating discussions concerning non-crossing configurations.

\section{Duality with trees and exploration of geodesics}
\subsection{Duality with trees}
 \label{sec:duality}

We briefly recall the
formalism of discrete plane trees which can be found in \cite{LG05} for
example. Let $\N=\{1,2\ldots\}$ be the set of positive integers and let $ \mathcal{U}$ be the set of labels
$$ \mathcal{U}=\bigcup_{n=0}^{\infty} (\N)^n,$$
where by convention $(\N)^0=\{\varnothing\}$. An element of $ \mathcal{U}$ is
a sequence $u=u_1 \cdots u_m$ of positive integers, and we set
$|u|=m$, which represents the generation, or height, of $u$. If $u=u_1
\cdots u_m$ and $v=v_1 \cdots v_n$ belong to $ \mathcal{U}$, we write $uv=u_1
\cdots u_m v_1 \cdots v_n$ for the concatenation of $u$ and $v$. A  \emph{plane tree} $\tau$ is then a finite or infinite subset of $
\mathcal{U}$ such that:
\begin{itemize}
\item[1.] $\varnothing \in \tau$,
\item[2.] if $v \in \tau$ and $v=uj$ for some $j \in \N$, then $u
\in \tau$,
\item[3.] for every $u \in \tau$, there exists an integer $k_u(\tau)
\geq 0$ (the number of children of $u$) such that, for every $j \in \N$, $uj \in \tau$ if and only
if $1 \leq j \leq k_u(\tau)$.
\end{itemize}
In the following,  \emph{tree} will always mean plane tree.  We will view each vertex of a tree $\tau$ as an individual of
a population whose $\tau$ is the genealogical tree.  The vertex $\varnothing$ is the ancestor of this population and {is} called the root. Every vertex $u \in \tau$ of degree $1$ is then called a leaf and the number of leaves of  $\tau$ is denoted by $\lambda( \tau)$. 
Last, for all $u, v \in \tau$, we denote by $\llbracket u,v \rrbracket$ the discrete geodesic path between $u$ and $v$ in $\tau$.

If $\tau$ is a plane tree, we denote by $\tau^\bullet$ the tree obtained from $\tau$ by attaching a leaf to the bottom of the root of $\tau$, and by rooting the resulting tree at this new leaf. Formally, we set $\tau^\bullet = \{ \varnothing\} \cup \{1u, u \in \tau\}$, and say that $\tau^\bullet$  is a planted tree.

\bigskip

For $n \geq 3$, we denote by $ \mathbf{D}_{n}$ the set of all the dissections of $\mathcal{P}_{n}$, and let $$ \overline{k}= \exp\left(\frac{-2  \mathrm{i}{k} \pi}{n}\right), \qquad 
\ 0 \leq k \leq n-1,$$ be the vertices of any dissection of $ \mathbf{D}_{n}$ {(the dependence in $n$ is implicit)}. {Given a dissection $ \mathcal{D} \in \mathbf{D}_n$, we construct a rooted plane tree as follows: Consider the dual graph
of $ {\mathcal{D}}$, obtained by placing a vertex inside each face of
$ \mathcal{D}$ and outside each side of the polygon  $\mathcal P_{n}$ and by
joining two vertices if the corresponding faces share a common edge,
thus giving a connected graph without cycles. This plane tree is rooted at the leaf adjacent to the edge $(\overline{0}, \overline{n-1})$  and is denoted  by $\phi( \mathcal{D})^\bullet$. Note that the root of $\phi( \mathcal{D})^\bullet$ has a unique child. Re-rooting the tree at this unique child and removing the former root and its adjacent edge gives a tree $\phi( \mathcal{D})$ with no vertex with exactly one child, whose planted version is  $\phi( \mathcal{D})^\bullet$. See Fig.\,\,\ref{fig:dual} below.

\begin{figure*}[h!]
\begin{center}
\includegraphics[height=5cm]{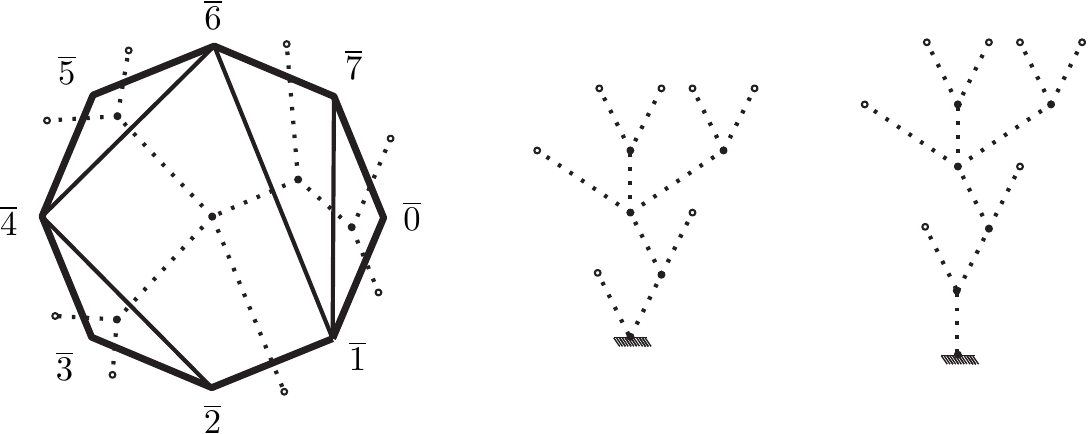}
\caption{\label{fig:dual} A dissection  $ \mathcal{D}$ of $ \mathcal{P}_{8}$ and its associated trees $ {\phi}(\mathcal{D})$ and $ \phi( \mathcal{D})^\bullet$. }
\end{center}
\end{figure*}

For $n \geq 3$, it is easy to see that the  application $  \phi$ is a bijection between $\mathbf{D}_{n}$ and the set of all  plane trees with $n-1$ leaves such that there is no vertex with exactly one child. For symmetry reasons, it will be more convenient to work with the planted tree $\phi( \mathcal{D})^\bullet$ rather than $\phi( \mathcal{D})$  (see e.g. \eqref{eq:reroot} below). However,  we also consider $\phi( \mathcal{D})$ because of its simple probabilistic description. If $\rho$ is a probability measure on $\mathbb Z_+$ such that $\rho(1)<1$,  the law of the Galton--Watson tree with offspring distribution $\rho$ is denoted by $ \mathsf{GW}_\rho$.

\begin{proposition}[{\cite{Kor12}, see also \cite{CKdissections}}] \label{prop:GW} Let $\mu$ be a probability distribution over $\{0,2,3,4\ldots\}$ of mean $1$. For every $n$ such that $ \mathsf{GW}_{\mu}( \lambda(\tau)=n-1)>0$, the dual tree $ \phi( \mathcal{D}_{n}^\mu)$ of a random dissection distributed according to $ \mathbb{P}_{n}^\mu$ is distributed according to $ \mathsf{GW}_{\mu}(  \, \cdot \, \mid \lambda(\tau)=n-1).$
\end{proposition}

This result explains the factor $c_{ \mathrm{tree}}( \mu)$ in the scaling constant $c(\mu)$ appearing in Theorem \ref{thm:main}. Indeed, {if we further} assume that $\mu$ has finite variance $\sigma^2$, then from \cite{Riz11,Kor12}, a $ \mathsf{GW}_{\mu}$ tree conditioned on having $n$ leaves and scaled by  $n^{-1/2}$ converges in distribution towards $c_{ \mathrm{tree}}(\mu)   \cdot  \mathcal{T}_{ \mathbf{e}}$ as $n \rightarrow \infty$. This is mainly due to the fact that  a $ \mathsf{GW}_{\mu}$ tree conditioned on having $n$ leaves is very close to a $ \mathsf{GW}_{\mu}$ tree conditioned on having $\mu_{0}^{-1}n$ vertices (see \cite{Kor12}), combined with the well-known result of Aldous 
 on the convergence of a $ \mathsf{GW}_{\mu}$ tree conditioned on having $n$ vertices and scaled by  $n^{-1/2}$, towards  $2 \sigma^{-1} \cdot \mathcal{T}_{ \mathbf{e}}$.  Hence
$$n^{-1/2} \cdot  \phi( \mathcal{D}_{n}^\mu)   \quad\mathop{\longrightarrow}^{(d)}_{n \rightarrow \infty} \quad   c_{ \mathrm{tree}}( \mu) \cdot \mathcal{T}_{ \mathbf{e}},$$ in distribution for the Gromov--Hausdorff topology. Obviously, the same statement holds when $  \phi( \mathcal{D}_{n}^\mu)$ is replaced by $ \phi( \mathcal{D}_{n}^{\mu})^\bullet$. \medskip

\noindent \textbf{Remark.} \quad  The criticality condition on $\mu$ and the fact that $ \mu$ is a probability measure are  not as restrictive as it could appear. Indeed, starting from  a sequence $(\mu_{i})_{i \geq 2}$ of nonnegative real numbers (recall that the definition of $\mathbb{P}_{n}^{\mu}$ does not involve $ \mu_{0}$ nor $ \mu_{1}$), one can easily build a critical probability measure $\nu$ such that $ \mathbb{P}_{n}^{\nu}= \mathbb{P}_{n}^{\mu}$,  provided that there exists $\lambda>0$ such that $\sum_{i \geq 2} i \lambda^{i-1} \mu_i=1$ (for example, such a $\lambda$ always exists when $ \sum_{i \geq 2} i\mu_{i} \in [1,\infty)$, but additional assumptions are needed otherwise). Indeed, in that case,  set
 $$ \nu_{0}= 1- \sum_{i \geq 2} \lambda^{i-1}\mu_{i},  \qquad \nu_{1}=0, \qquad \nu_{i} = \lambda^{i-1}\mu_{i} \qquad   {( i \geq 2)},$$
 which defines a critical probability measure. Then it is easy to check (see e.g.\,\,the proof of \cite[Proposition 2.3]{CKdissections}) that $ \mathbb{P}_{n}^{\nu}= \mathbb{P}_{n}^{\mu}$.

\subsection{Geodesics in the dissection} \label{sec:geoexplo}

Now that we have associated a dual tree with each dissection, we shall see how to find the geodesics in the dissection using the geodesics in its dual tree.

We fix a dissection $ \mathcal{D} \in \mathbf{D}_{n}$. By the rotational invariance of the model we shall only describe geodesics in $ \mathcal{D}$ from the vertex $ \overline{0}$. 
Let $ \varnothing = \ell_{0}, \ell_{1}, \ldots , \ell_{n-1}$ be the $n$ leaves of $ {\phi}( \mathcal{D})^\bullet$ in clockwise order. Our first observation states that the geodesics in the dissection stay very close to their dual geodesics in the tree.

\begin{proposition} \label{prop:geo1}For every $ k \in \{0,1, \ldots , n-1\}$, the dual edges of a geodesic path from $ \overline{0}$ to $ \overline{k}$ in $ \mathcal{D}$ are {all} {adjacent to} the geodesic path $\llbracket \ell_{0}, \ell_{k}\rrbracket$ in ${ \phi}( \mathcal{D})^\bullet$.
\end{proposition}

\proof The proof is clear on a drawing (see Fig.\,\ref{fig:geodesic}, where $k=12$ and where the geodesic $\llbracket \ell_{0}, \ell_{k}\rrbracket$ in ${ \phi}( \mathcal{D})^\bullet$ is in bold). A geodesic in $ \mathcal{D}$ going from $ \overline{0}$ to $ \overline{k}$ will only use edges of $ \mathcal{D}$ that belong to the faces crossed by the geodesic path $\llbracket \ell_{0}, \ell_{k}\rrbracket$ in $ \phi(\mathcal{D})^\bullet$ (which are the white faces in Fig.\,\ref{fig:geodesic}). Indeed, it is easy to see that such a geodesic in $ \mathcal{D}$ will never enter the other faces 
(which are shaded in gray in Fig.\,\ref{fig:geodesic}), since any one of these faces is separated from the rest by a single edge of $ \mathcal{D}$. \endproof

\bigskip

\noindent \textbf{A local iterative construction.} \quad We now detail how to 
obtain a geodesic  going from $ \overline{0}$ to $ \overline{k}$ 
in $ \mathcal{D}$ by an iterative ``local'' 
construction 
 along the 
geodesic $\llbracket \ell_{0}, \ell_{k}\rrbracket$ in the dual tree $ \phi(\mathcal{D})^\bullet$ (note that there may exist several geodesics  going from $ \overline{0}$ to $ \overline{k}$ 
in $ \mathcal{D}$, our procedure only produces one of them). Before doing so, let us make a couple of observations and introduce 
 a piece of  notation.

 \medskip 

Fix $ k \in \{1, \ldots n-1\}$. Let $ h$ be the number of edges of $ \llbracket \ell_{0}, \ell_{k}\rrbracket$ ($h$ is the height of $ \ell_{k}$ in $ \phi( \mathcal{D})^\bullet$) 
 and denote by $w_{0},w_{1}, \ldots, w_{h}$  the vertices of $ \llbracket \ell_{0}, \ell_{k}\rrbracket$ (ordered in increasing height). Next, 
for every $0 \leq i \leq h-1$, let $e_{i}$ be the edge of $ \mathcal{D}$ which is dual to the edge $w_{i}w_{i+1}$ of $ \phi(\mathcal{D})^\bullet$. For $0 \leq i \leq h-1$, the endpoint of $e_{i}$ which is located on the left, resp.\,right, of $ \llbracket \ell_{0}, \ell_{k}\rrbracket$ (when oriented from $\ell_{0}$ to $ \ell_{k}$) is denoted by $e_{i}^{ \mathrm{L}}$, resp.\,$e_{i}^{ \mathrm{R}}$ (note that one may have $e_{i+1}^{ \mathrm{R}}=e_{i}^{ \mathrm{R}}$, and similarly for $ \mathrm{L}$). See Fig. \ref{fig:geodesic}.

Consider now $ \mathcal{G}= \{ \overline{0}=x_{0}, x_{1}, \ldots , x_{m}= \overline{k}\}$ the set of all the vertices of a geodesic in $ \mathcal{D}$ going from $ \overline{0}$ to $ \overline{k}$. An easy geometrical argument shows that for every $i \in \{0, \ldots , h-1\}$, if the  edge $e_{i}$ together with its endpoints is removed from $ \mathcal{D}$, then the vertices $ \overline{0}$ and $ \overline{k}$ become disconnected (or absent) in $\mathcal{D}$. Hence, for every $0 \leq i \leq h-1$,  at least one of the endpoints $e_{i}^R$ or $e_{i}^L$ of the edge $e_{i}$ belongs to $ \mathcal{G}$. Furthermore, the geodesic $ \mathcal{G}$ visits  $e_{0}, e_{1}, \ldots , e_{h-1}$ in this order (we say that $ \mathcal{G}$ visits an edge $e$ if one of the endpoints of $e$ belongs to $ \mathcal{G}$) and for every $1 \leq i \leq h-1$, after $ \mathcal{G}$ has  visited $e_{i}$, $ \mathcal{G}$ will not visit $e_{j}$ for every $0 \leq  j < i$.  Finally, we denote by $ \mathrm{d_{ \mathcal{D}}}$ the graph distance in the dissection $ \mathcal{D}$.
\medskip 

\noindent \textsc{The algorithm $ \mathsf{Geod}( \overline {k})$.} We now present 
an algorithm called $ \mathsf{Geod}( \overline {k})$  that constructs  ``step-by-step'' a geodesic in $\mathcal D$ going from $ \overline{0}$ to $ \overline{k}$. Formally, we shall iteratively construct a path $\mathscr{P}= \{ y_{0}, y_{1}, \ldots \}$ 
of vertices going from $ \overline{0}$ to $ \overline{k}$ 
together with a sequence of integers $(s_{i} : 0 \leq i \leq h)$ such that  the cardinal of  $\mathscr {P}$ is $s_{h}+1$ and, for every $ i \in \{ 0, 1, \ldots , h-1\}$, $$s_{i} = \inf\{ j \geq 0 : y_{j}  = e^R_{i} \textrm{ or } e^L_{i}\}$$ 
(this infimum will always be reached). The induction procedure will be on $i  \in \{ 0, 1, \ldots , h\}$.
For $i \leq h-1$, we will not always know at stage $i$ if $y_{s_{i}} = e_{i}^{\mathrm L}$ or $y_{s_{i}} = e_{i}^{\mathrm R}$. In the cases when this is known, we define the position $p_{i} \in \{\mathrm{L} , \mathrm{R}\}$ through $y_{s_{i}} = e_{i}^{p_{i}}$ and say that the position is ``determined''. Otherwise we set $p_i={\mathrm U}$ and say that the position is ``undetermined".

The induction then proceeds as follows. First, set $ y_{0}=  \overline {0}$, so that $s_{0}=0$ and $p_{0}= \mathrm{L}$. Also, for reasons that will appear later, let $ \mathcal{I}$ be an empty set. Then, recursively for $i \in \{0,1, \ldots , h-2\}$,
assume that  $  \{s_{0}, s_{1},\ldots, s_{i}\}$ and  $  \{p_{0}, p_{1},\ldots, p_{i}\}$ have been constructed, as well as $\{y_{s_0},y_{s_1},\ldots,y_{s_i} \}$ in the cases where $p_i \in \{\mathrm{L} , \mathrm{R}\}$.  Denote by $g_{i}$ the number of edges of $ \phi( \mathcal{D})^\bullet$ adjacent to $w_{i+1}$ that are strictly on the left of $ \llbracket \ell_{0}, \ell_{k}\rrbracket$ and let $E^{g}_{i}$ be set of edges in $ \mathcal{D}$ that are dual to those edges. Similarly, let $d_{i}$ be the number of edges adjacent to $w_{i+1}$ that are strictly on the right of $ \llbracket \ell_{0}, \ell_{k}\rrbracket$ and let $E^{d}_{i}$ be set of edges in $ \mathcal{D}$ that are dual  to those edges.

\begin{figure}[!h]
  \begin{center}
  \includegraphics[width=1\linewidth]{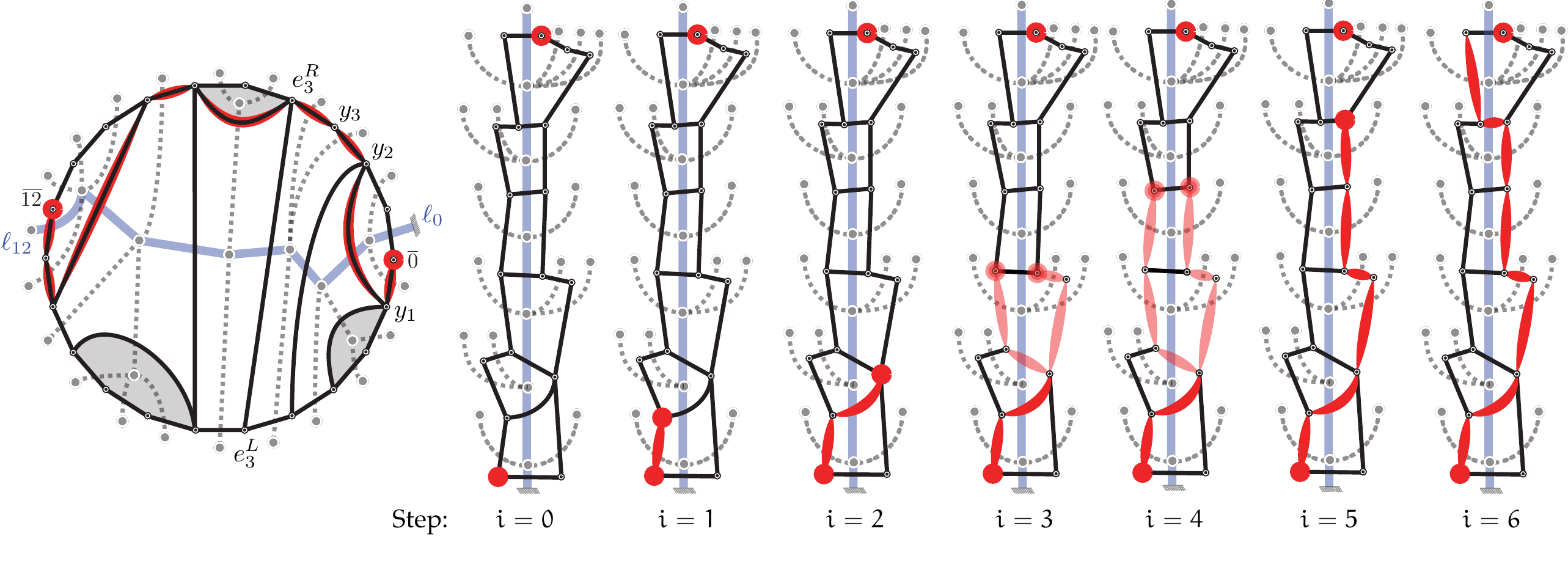}
  \caption{\label{fig:geodesic}Illustration of the steps of the algorithm constructing a geodesic between $ \overline {0}$ and $ \overline{12}$ in $ \mathcal{D}$.  The undetermined steps are in light color.}
  \end{center}
  \end{figure}
  
We now want to build a shortest path in $\mathcal D$ from the current position $y_{s_i} \in \{e_{i}^{ \mathrm{L}},e_{i}^{ \mathrm{R}}\}$ to $\overline k$. 
 In that aim, we have to decide whether  $ y_{s_{i+1}}=e_{i+1}^{ \mathrm{L}}$ or $y_{s_{i+1}}=e_{i+1}^{ \mathrm{R}}$ or if we have to wait for a further step to decide whether the right or left position is best.
  Note that $|\mathrm{d_{ \mathcal{D}}}(e_{i+1}^{ \mathrm{L}}, \overline{k})-\mathrm{d_{ \mathcal{D}}}(e_{i+1}^{ \mathrm{R}}, \overline{k})| \leq 1$ since $ \mathrm{d_{ \mathcal{D}}}(e_{i+1}^{ \mathrm{L}},e_{i+1}^{ \mathrm{R}}) =1$. Hence in order to choose whether $ y_{s_{i+1}}=e_{i+1}^{ \mathrm{L}}$ or $y_{s_{i+1}}=e_{i+1}^{ \mathrm{R}}$, we have to compare  $ \mathrm{d}_{ \mathcal{D}}(y_{s_{i}}, e_{i+1}^{L})$ with  $ \mathrm{d}_{ \mathcal{D}}(y_{s_{i}}, e_{i+1}^{R})$. 
  
  There are five different cases:
  \begin{itemize} 
  \item \textsc{The position stays determined and stays on the same side of $\llbracket \ell_{0}, \ell_{k}\rrbracket$:} If $ {p}_{i} =  \mathrm{L}$ and $g_{i} \leq d_{i}$. In this case (in Fig. \ref {fig:geodesic}, this happens for $i=0$), we have $ \mathrm{d}_{ \mathcal{D}}(e_{i}^{ \mathrm{L}}, e_{i+1}^{ \mathrm{L}}) < \mathrm{d}_{ \mathcal{D}}(e_{i}^{ \mathrm{L}}, e_{i+1}^{ \mathrm{R}})$, hence we 
  add to $ \mathscr{P}$ the vertices visited when walking along the  edges of $E_{i}^g$  (here and later, we do not add a vertex to $ \mathscr {P}$  if it is already present in $ \mathscr {P}$) and set $$s_{i+1} = s_{i} + g_{i} \quad  \mbox{and} \quad  y_{s_{i+1}}= e_{i+1}^{ \mathrm{L}} \  \text{ (hence }p_{i+1}= \mathrm{L}).$$
\quad The case  $ {p}_{i} =  \mathrm{R}$ and $d_{i} \leq g_{i}$ is similar: in this case, we add to $ \mathscr{P}$ the vertices visited when walking along the edges of $ E ^{d}_{i}$, and set $s_{i+1} = s_{i} + d_{i}$ and $ y_{s_{i+1}}= e_{i+1}^{ \mathrm{R}}$ (hence $p_{i+1}= \mathrm{R}$).
 
  \item \textsc{The position stays determined and changes sides:} If $p_{i} = \mathrm{L}$ and $d_{i}+1 < g_{i}$. In this case (in Fig. \ref {fig:geodesic}, this happens for $i=1$) we have $ \mathrm{d}_{ \mathcal{D}}(e_{i}^{ \mathrm{L}}, e_{i+1}^{ \mathrm{R}}) < \mathrm{d}_{ \mathcal{D}}(e_{i}^{ \mathrm{L}}, e_{i+1}^{ \mathrm{L}})$. We thus add to $ \mathscr{P}$  the vertex $ e_{i}^{R}$ as well as the vertices visited when walking along the edges of $ E ^{d}_{i}$. Then we set $s_{i+1}=s_{i} + 1 + d_{i}$ {and}   $y_{s_{i+1}}= e_{i+1}^{ \mathrm{R}}$ (hence $p_{i+1} = \mathrm{R}$).
  
\quad The case $p_{i} = \mathrm{R}$ and $g_{i}+1 < d_{i}$ is symmetric (in Fig. \ref {fig:geodesic}, this happens if $i=5$).

  \item \textsc{The position becomes undetermined:} If  $p_{i} = \mathrm{L}$ and $1+ d_{i} = g_{i}$, or if $p_{i} = \mathrm{R}$ and $1+ g_{i} = d_{i}$.
  In these cases (in Fig. \ref {fig:geodesic}, this happens for $i=2$),  we have $ \mathrm{d}_{ \mathcal{D}}(e_{i}^{p_{i}}, e_{i+1}^{ \mathrm{L}}) = \mathrm{d}_{ \mathcal{D}}(e_{i}^{ p_{i}}, e_{i+1}^{ \mathrm{R}})$ hence we cannot decide right away if $ y_{s_{i+1}}=e_{i+1}^{ \mathrm{L}}$ or $y_{s_{i+1}}=e_{i+1}^{ \mathrm{R}}$. We thus need to use the additional undetermined state $  \mathrm{U}$, and set $p_{i+1}= \mathrm{U}$. In this cases, we add no new vertices to the set $ \mathscr{P}$, but instead add  to the set $ \mathcal{I}$  the edges of $ E ^{g}_{i}$ and   $ E ^{d}_{i}$ (the set $ \mathcal{I}$ contains the so-called undetermined edges).
Moreover, in both cases, we set  $s_{i+1}=s_{i} + 1+ d_{i}=s_i+g_{i}$.
  
  \item \textsc{The position stays undetermined:} If $p_i = \mathrm{U}$ and $d_{i}=g_{i}$. In this case (in Fig. \ref {fig:geodesic}, this happens for $i=3$), since the position $p_{i}$ is either left or right, the distance between $y_{s_{i}}$ and  $e_i^{\mathrm{R}}$ or the distance between $y_{s_{i}}$ and  $e_i^{\mathrm{L}}$ can be chosen to be $d_{i}=g_{i}$. We thus stay undetermined and set $p_{i+1}= \mathrm{U}$ {and}    $s_{i+1}= s_i + d_{i}$. Furthermore, we add no new vertices to the set $ \mathscr{P}$, but add instead  the edges of $ E ^{g}_{i}$ and  $ E ^{d}_{i}$ to the set $ \mathcal{I}$.
  
    \item \textsc{The position becomes determined:} If $p_i = \mathrm{U}$ and $d_{i} \ne g_{i}$. In this case (in Fig. \ref {fig:geodesic}, this happens for $i=4$), if $d_{i}<g_{i}$,   then $\mathrm{d}_{ \mathcal{D}}(e_{i}^{ \mathrm{R}}, e_{i+1}^{ \mathrm{R}}) < \mathrm{d}_{ \mathcal{D}}(e_{i}^{ \mathrm{L}}, e_{i+1}^{ \mathrm{L}})$ and we set  $s_{i+1}=s_{i}+d_{i}$ and $y_{s_{i+1}}= e_{i+1}^{ \mathrm{R}}$ (hence $p_{i+1}= \mathrm{R}$). We then add to $ \mathscr {P}$ all the vertices visited when crossing the undetermined edges of $ \mathcal{I}$ which are on the right of $ \llbracket \ell_{0}, \ell_{k}\rrbracket$, and now set $ \mathcal{I}= \emptyset$.
    
\quad The case $g_{i}<d_{i}$ is symmetric.\end{itemize}
\textsc{Last step ($i=h-1$).} If $p_{i}= \mathrm{R}$, we set $s_{i+1}=s_{i}$. If $p_{i}= \mathrm{L}$ (in Fig. \ref {fig:geodesic}, this happens for $i=6$), we add the endpoints of $e_{i}^{R}$ to $ \mathscr {P}$ and set $s_{i+1}=s_{i}+1$. Finally, if $p_{i}= \mathrm{U}$, we add to $ \mathscr {P}$ the vertices visited when walking along the edges of $E_{i}^d$   and set $s_{i+1}=s_{i}$. 

\bigskip

This finishes the construction of the path $ \mathscr {P}$. The following result  should be clear (see Fig. \ref{fig:geodesic}):

\begin{proposition} \label {prop:geod} The path $ \mathscr {P}$ constructed by  $ \mathsf{Geod}( \overline {k})$ is a geodesic path in $ \mathcal{D}$ from $ \overline{0}$ to $ \overline{k}$ whose length is $s_{h}$. \end{proposition}

In the sequel, we will only be interested in the length $s_h$ of this specific geodesic going from $ \overline {0}$ to $ \overline k$. 
Recall that $h$ is the height of $ \ell_{k}$ in $ \phi( \mathcal{D})^\bullet$. The explicit construction of $ \mathscr{P}$ implies that the sequence $(g_{n},d_{n},p_{n},s_{n})_{0 \leq n \leq h-1}$ obtained when running $ \mathsf{Geod}( \overline {k})$ satisfies $s_{0}=0$, $p_{0}=L$, and then  for every $0 \leq n \leq h-2$, setting $\Delta s_{n+1} = s_{n+1}-s_{n}$:
\begin{enumerate}

\item[$\bullet$] If $p_n=\mathrm R$, \quad  $\begin{array}{ll}\text{if }d_{n}< g_{n}+1 & \text{then } (\Delta s_{n+1},p_{n+1})=(d_{n},\mathrm R) 
\\ \text{if }d_{n}>g_{n}+1  & \text{then } (\Delta s_{n+1},p_{n+1})=(g_{n}+1,\mathrm L) 
\\ \text{if }d_{n}=g_{n}+1  & \text{then } (\Delta s_{n+1},p_{n+1})=(d_{n},\mathrm U); \end{array}$

\item[$\bullet$] If $p_n=\mathrm L$, \quad $\begin{array}{ll}\text{if }g_{n}< d_{n}+1 & \text{then } (\Delta s_{n+1},p_{n+1})=(g_{n},\mathrm L) 
\\ \text{if }g_{n}>d_{n}+1  & \text{then } (\Delta s_{n+1},p_{n+1})=(d_{n}+1,\mathrm R) 
\\ \text{if }g_{n}=d_{n}+1  & \text{then } (\Delta s_{n+1},p_{n+1})=(g_{n},\mathrm U); \end{array}$
\item[$\bullet$] If $p_n=\mathrm U$,  \quad $\begin{array}{ll}\text{if }d_{n}<g_{n} & \text{then } (\Delta s_{n+1},p_{n+1})=(d_{n},\mathrm R) \\ \text{if }d_{n}>g_{n}  & \text{then } (\Delta s_{n+1},p_{n+1})=(g_{n},\mathrm L) \\ \text{if }d_{n}=g_{n}  & \text{then } (\Delta s_{n+1},p_{n+1})=(d_{n}, \mathrm U).\end{array}$ 
\end{enumerate}
Now set $H_{\phi( \mathcal{D})^\bullet}( \ell_{k})= s_{h-1}$. Since $|s_{h}-s_{h-1}| \leq 1$ by construction,  we get from Proposition \ref {prop:geod} that   \begin{align} \label{eq:procheun} \big|d_{ \mathcal{D}}(\overline{0}, \overline{k})- H_{ \phi(\mathcal{D})^\bullet}(\ell_{k})\big| &\leq 1. \end{align} 

For later use, we now extend the definition of $H_{ \tau}(u)$ to general  trees $ \tau$ and every vertex $ u \in \tau$ (not only leaves). To this end, denote by $ \tau_{[u]}$ the subtree of $\tau$ formed by the vertices of $ \llbracket \varnothing, u\rrbracket$ together with the children of vertices belonging to $ \rrbracket \varnothing, u\llbracket$. Note that when $ \tau$ is a finite tree and $ u \in \tau$ is a leaf, then by the previous discussion $H_{\tau}(u)$ only depends on $\tau_{[u]}$. Hence, for $ \tau$  a possibly infinite tree and $u$ any vertex of $\tau$, we can set 
$$H_{ \tau}(u) := H_{ \tau_{[u]}}(u) \text{ when } u \ne \varnothing, \text{ and } H_{ \tau}(\varnothing)=0.$$
\section{A Markov chain} \label{sec:makov}

In the remaining sections, $ \mu$ denotes a probability distribution on $\{0,2,3, \ldots \}$ with mean 1 and such that  $ \sum_{i \geq 0} e^ { \lambda  i} \mu_{i}< \infty$ for some $\lambda>0$. 
To prove Theorem \ref {thm:main}, it will be important to describe the asymptotic behavior of the length of a typical geodesic of the random dissection $\mathcal D_n^{\mu}$  as $n \rightarrow \infty$.  To this end, the first step is to understand the behavior of the algorithm $ \mathsf{Geod}$ when run on
the spine of the critical Galton--Watson tree conditioned to survive. This can informally be seen as the ``unconditioned version'', where we gain some independence (specifically, the variables $(g_i,d_i)$ of the last section become i.i.d.). In that setting, the algorithm $ \mathsf{Geod}$ yields a true Markov chain whose asymptotic behavior is studied in Section \ref{sec:mc}. The second step, carried out later in Section \ref {sec:proof}, consists in going back to the ``conditioned version''  $ \GW_{ \mu}$.

\subsection{The critical Galton--Watson tree conditioned to survive}
 \label{section:GWinfinite}

If $\tau$ is a tree and $k \geq 0$, we let $[\tau]_{k}=  \{u \in \tau : \, |u| \leq k\}$ denote the subtree of $\tau$ composed by its first $k$ generations. 
We denote by $T_{n}$ a Galton--Watson tree with offspring distribution $\mu$, conditioned on
having height at least $n \geq 0$. Kesten \cite[Lemma 1.14]{Kes86} showed that for every $k \geq 0$, the convergence
  \begin{eqnarray*} \left[T_{n}\right]_{k} & \quad \xrightarrow[n\to\infty]{(d)} \quad & \left[T_{\infty}\right]_{k}, \end{eqnarray*}
  holds in distribution, where $T_{\infty}$ is a random infinite plane tree called the critical $ \GW_{ \mu}$ tree conditioned to survive. Since we mainly consider planted trees, let us describe the law of $T_{\infty}^\bullet$.  We follow
\cite{Kes86,LP10}. 
First let $\mu^\star$ be the size-biased
distribution of $\mu$, defined by $ \mu^\star_{k} = k \mu_{k}$
for every $k \geq 0$. Next, let $(C_i)_ {i \geq 1}$ be a sequence of i.i.d. random variables distributed according to 
$\mu^\star$ and let $C_{0}=1$. Conditionally on $(C_i)_ {i \geq 0}$, let $(V_{i+1})_ {i \geq 0}$ be a sequence of independent random variables such that   $V_{k+1}$ is uniformly distributed over $\{ 1, 2, \ldots , C_{k}\}$, for every $k \geq 0$.  Finally, let $W_{0}= \varnothing$ and $W_k=V_{1} V_{2} \ldots V_{k}$ for $k \geq 1$.

The infinite tree $T_{ \infty}^\bullet$ has a unique spine, that is a unique infinite path
$(W_{0}, W_{1}, W_{2},  \ldots)$ and, for $k \geq 0$, $W_{k}$ has $C_k$ children. Then,
conditionally on $(V_i)_ {i \geq 1}$ and $(C_i)_ {i \geq 0}$, all children of $W_k$ except $W_{k+1}$, $\forall k \geq 1$, have independent $ \GW_{ \mu}$ descendant trees, see Fig.\,\ref{critical}.  

\begin{figure}[!h]
  \begin{center}
  \includegraphics[height=5cm]{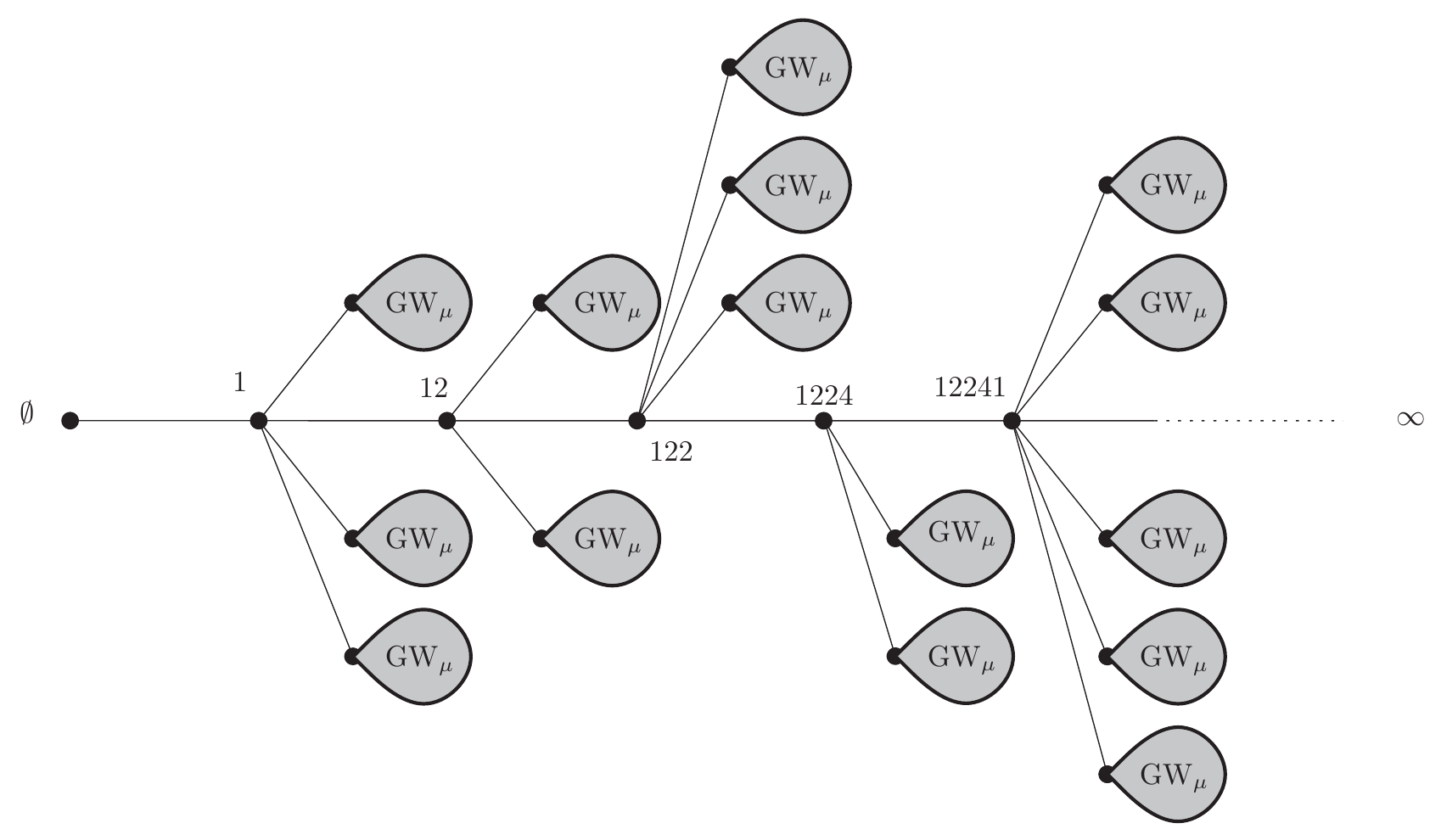}
  \caption{\label{critical} An illustration of $T_{\infty}^\bullet$.}
  \end{center}
  \end{figure}

The following result states a useful relation between a standard $\mathsf{GW}_\mu$  and the infinite version $T^{ \bullet}_\infty$ (see e.g. \cite[Chapter 12.1]{LP10} for a proof when $T^{ \bullet}_\infty$ is replaced by $T_\infty$). We let $\mathbb T$ denote the set of all discrete plane trees.

\begin{proposition} \label{prop:biais} For every measurable function $ F:  \mathbb {T} \times \mathcal{U} \rightarrow \R_{+}$ and for every $n \geq 0$, we have
{$$ \GW_{ \mu}\left[ { \sum_{u \in \tau^\bullet, |u|=n} F\big([\tau^\bullet]_{n},u\big)}\right]=  \mathbb{E}\Big [ F\big( [T_{\infty}^\bullet]_{n},W_{n} \big)\Big]$$}
\end{proposition}

\subsection {The Markov chain}
\label {sec:mc}
Recall the definition of $H$ at the end of Section \ref {sec:geoexplo}. Set $S_{0}=0$ and for $n \geq 1$, set
 \begin{eqnarray} \label{eq:snh} S_{n} = H_{T_{\infty}^\bullet}(W_{n+1}).  \end{eqnarray}
Informally, $(S_{n})_{n\geq 0}$ is the length process of a path of minimal length in the ``dual dissection'' of $T_{\infty}^{\bullet}$ starting from the root and running along the spine of $T^{ \bullet}_{\infty}$.  The goal of this section is to prove the almost sure convergence of $n^{-1}S_n$ towards $c_{ \mathrm{geo}}$ (where $c_{ \mathrm{geo}}$ is the second factor in the constant $c(\mu)$ of Theorem \ref{thm:main}) and then to establish large deviations estimates. These will be useful to deduce Theorem \ref{thm:main} in  Section \ref {sec:proof}.

 \medskip 

 By analogy with the notation of Section \ref{sec:geoexplo}, for $i \geq 0$, we let $G_{i} = V_{i+2}-1$ be the number of children of $W_{i+1}$ on the left of the spine and similarly we let $D_{i} = C_{i+1} -V_{i+2}$ be the number of children of $W_{i+1}$ on the right of the spine. We then build a Markov chain $(X_n,P_n)_{n \geq0}$ with values in $\mathbb Z_+ \times \{\mathrm{R,L,U} \}$ following the procedure of the Section \ref{sec:geoexplo}. 
Formally the evolution of this chain is given by the following rules.
First, $X_0=0$, $P_0=\mathrm L$. Next\begin{enumerate}
\item[$\bullet$] If $P_n=\mathrm R$, \quad $\begin{array}{ll}\text{if }D_{n}< G_{n}+1 & \text{then } (X_{n+1},P_{n+1})=(D_{n},\mathrm R) 
\\ \text{if }D_{n}>G_{n}+1  & \text{then } (X_{n+1},P_{n+1})=(G_{n}+1,\mathrm L) 
\\ \text{if }D_{n}=G_{n}+1  & \text{then } (X_{n+1},P_{n+1})=(D_{n},\mathrm U); \end{array}$

\item[$\bullet$] If $P_n=\mathrm L$, \quad 
$\begin{array}{ll}\text{if }G_{n}< D_{n}+1 & \text{then } (X_{n+1},P_{n+1})=(G_{n},\mathrm L) 
\\ \text{if }G_{n}>D_{n}+1  & \text{then } (X_{n+1},P_{n+1})=(D_{n}+1,\mathrm R) 
\\ \text{if }G_{n}=D_{n}+1  & \text{then } (X_{n+1},P_{n+1})=(G_{n},\mathrm U); \end{array}$
\item[$\bullet$] If $P_n=\mathrm U$, \quad $\begin{array}{ll}\text{if }D_{n}<G_{n} & \text{then } (X_{n+1},P_{n+1})=(D_{n},\mathrm R) \\ \text{if }D_{n}>G_{n}  & \text{then } (X_{n+1},P_{n+1})=(G_{n},\mathrm L) \\ \text{if }D_{n}=G_{n}  & \text{then} (X_{n+1},P_{n+1})=(D_{n}, \mathrm U). \end{array}$ 
\end{enumerate}
 From the discussion following Proposition \ref {prop:geod}}, we have $ \displaystyle S_{n} = X_{0}+ \cdots+ X_{n}$ for every $n \geq 0$.

 \bigskip
 
The transition probabilities from $(X_n,P_n)$ to $(X_{n+1},P_{n+1})$ only depend on the value of $P_n$.  The process $(S_n,P_n)_{n \geq 0}$ therefore belongs to the family of so-called \textit{Markov additive processes} (see e.g. \cite{Cinlar72}) and $(P_n)_{n \geq 0}$ is called its \textit{driving chain}. To simplify notation, set $\overline \mu_k=\sum_{i \geq k} \mu_i$ for $k \geq 0$. From the explicit distribution of $(C_i,V_{i+1})_{i \geq 1}$ (note that they are i.i.d) we easily calculate the transition probabilities of $(X_{n},P_{n})$: 
For all $i \geq 0$,
\begin{align*}
\mathbb P\left(X_{n+1}=i, P_{n+1}=\mathrm R \ | \ P_n= \mathrm R \right)&=\mathbb P\left(X_{n+1}=i, P_{n+1}=\mathrm L \ | \ P_n= \mathrm L \right)=\overline \mu_{2i+1} \\
\mathbb P\left(X_{n+1}=i, P_{n+1}=\mathrm L \ | \ P_n=\mathrm R \right)&=\mathbb P\left(X_{n+1}=i, P_{n+1}=\mathrm R \ | \ P_n=\mathrm L \right)=\overline \mu_{2i+1}\mathbbm{1}_{\{i \geq 1\}} \\
\mathbb P\left(X_{n+1}=i, P_{n+1}=\mathrm U \ | \ P_n=\mathrm R \right)&=\mathbb P\left(X_{n+1}=i, P_{n+1}=\mathrm U \ | \ P_n=\mathrm L \right)=\mu_{2i}\mathbbm{1}_{\{i \geq 1\}}
\end{align*}
and, 
\begin{eqnarray*}
&&\mathbb P\left(X_{n+1}=i, P_{n+1}=\mathrm R \ | \ P_n=\mathrm U \right)=\mathbb P\left(X_{n+1}=i, P_{n+1}=\mathrm L \ | \ P_n=\mathrm U \right)=\overline \mu_{2i+2} \\
&&\mathbb P\left(X_{n+1}=i, P_{n+1}=\mathrm U \ | \ P_n=\mathrm U \right)=\mu_{2i+1}.
\end{eqnarray*}
Note that  the right and left positions $ \mathrm{R}$ and $ \mathrm{L}$ play symmetrical roles. Hence, with a slight abuse of notation, we will consider from now on that $P_n$ can take only two values: $ \mathrm D$ (for Determined) or $ \mathrm U$, with the convention that $P_n=\mathrm D$ if and only if $P_n \in \{\mathrm L,\mathrm R\}$. From the previous calculations, we thus get for every $i \geq 0,$
\begin{eqnarray*}
\mathbb P\left(X_{n+1}=i, P_{n+1}=\mathrm D \ | \ P_n= \mathrm D \right) &=& \overline \mu_{2i+1} +\overline \mu_{2i+1}\mathbbm{1}_{\{i \geq 1\}}\\
\mathbb P\left(X_{n+1}=i, P_{n+1}=\mathrm U \ | \ P_n=\mathrm D \right)&=&\mu_{2i}\mathbbm{1}_{\{i \geq 1\}} \\
\mathbb P\left(X_{n+1}=i, P_{n+1}=\mathrm D \ | \ P_n=\mathrm U \right)&=&2\overline \mu_{2i+2} \\
\mathbb P\left(X_{n+1}=i, P_{n+1}=\mathrm U \ | \ P_n=\mathrm U \right)&=&\mu_{2i+1}.
\end{eqnarray*}
Recall that  $\mu_{2\mathbb Z_+}=\sum_{i \geq 0} \mu_{2i}$ and let $\mu_{2\mathbb N}=\sum_{i \geq 1} \mu_{2i}$ and $\mu_{2\mathbb N+1}=\sum_{i \geq 0} \mu_{2i+1}$. The previous discussion leads to the following description of the driving chain $(P_{n})$. 
\begin{lemma}
\label{LemmaDriving}
The driving chain $(P_n)$ has the following transition probabilities:
\begin{eqnarray*}
&& \mathbb P\left(P_{n+1}=\mathrm D \ | \ P_n= \mathrm D \right)= \mu_{2 \mathbb N+1}+\mu_0 =1-\mathbb P\left(P_{n+1}=\mathrm U \ | \ P_n= \mathrm D \right) \\
&& \mathbb P\left(P_{n+1}=\mathrm D \ | \ P_n= \mathrm U \right)= \mu_{2 \mathbb Z_+}=1-\mathbb P\left(P_{n+1}=\mathrm U \ | \ P_n= \mathrm U \right).
\end{eqnarray*}
This chain is irreducible and aperiodic if and only if $\mu_{2 \mathbb N}>0$. In this case, its stationary distribution $\pi$ is
$$
\pi(\mathrm D)=\frac{\mu_{2\mathbb Z_+}}{\mu _{2\mathbb Z_+}+ \mu _{2\mathbb N}}, \quad \quad \pi(\mathrm U)=\frac{\mu_{2\mathbb N}}{\mu _{2\mathbb Z_+}+ \mu _{2\mathbb N}}.
$$
\end{lemma}
In order to establish a strong law of large numbers for $(S_n)$, it is useful to introduce the mean of a typical step of the driving chain in the stationary state:
$$
c_{ \mathrm{\mathrm{geo}}}( \mu) \quad := \quad \mathbb E_{\pi}[X_1]=\sum_{i\geq 0}i \mathbb P\left(X_1=i \ | \ P_0= \mathrm D \right)\pi(\mathrm D) + i \mathbb P\left(X_1=i \ | \ P_0= \mathrm U \right)\pi(\mathrm U).
$$
Note that this also makes sense when $\mu_{2 \mathbb N}=0$ since $P_0= \mathrm D$. We now give an explicit expression of  $c_{ \mathrm{\mathrm{geo}}}( \mu) $ in terms of $ \mu$. Recall that $\sigma^2$ denotes the variance of $\mu$.

\begin{lemma} We have $ \displaystyle 
c_{\mathrm{\mathrm{geo}}}( \mu)=\frac{1}{4}\left(\sigma^2+\frac{\mu_0 \mu _{2 \mathbb Z_+}}{2 \mu_{2 \mathbb Z_+}-\mu_0} \right)$.\end{lemma}

\proof Note first that
 
\begin{eqnarray*}
c_{\mathrm{\mathrm{geo}}}( \mu)
&=& \frac{\left( \sum_{i \geq 0} i\overline \mu_{2i+1}+ \sum_{i \geq 0} i\overline \mu_{2i}   \right)\mu _{2\mathbb Z_+}+\left( \sum_{i \geq 0} i\overline \mu_{2i+1}+ \sum_{i \geq 0} i\overline \mu_{2i+2}\right)\mu_{2 \mathbb N}}{{\mu _{2\mathbb Z_+}+\mu_{2 \mathbb N}}}
\end{eqnarray*}
and then that 
\begin{equation*}
\label{Oddsum}
\sum_{i \geq 0} i\overline \mu_{2i+1} = \sum_{k \geq 1} \mu_k \sum_{i=0}^{[(k-1)/2]}i=\frac{1}{2}\sum_{k \geq 1} \mu_k \left[\frac{k-1}{2}\right] \left[\frac{k+1}{2}\right],
\end{equation*}
where $[r]$ denotes the largest integer smaller than $r \in \R$. Similarly, $$\sum_{i \geq 0} i\overline \mu_{2i} =\frac{1}{2}\sum_{k \geq 1} \mu_k \left[\frac{k}{2}\right] \left[\frac{k}{2}+1\right]$$ and since $\left[(k-1)/2\right] \left[(k+1)/2\right]  +\left[k/2\right] \left[k/2+1\right]$ is equal to $k^2/2$ when $k$ is even and $(k^2-1)/2$ when $k$ is odd, we finally get
$$
\sum_{i \geq 0} i\overline \mu_{2i+1}+ \sum_{i \geq 0} i\overline \mu_{2i}  =\frac{1}{4}\left(\sum_{k \geq 1} k^2 \mu_k-\mu_{2\mathbb N+1}\right)=\frac{\sigma^2+1-\mu_{2 \mathbb N+1}}{4}=\frac{\sigma^2+\mu_{2 \mathbb Z_+}}{4}.
$$
Similarly (recall that $\mu_1=0$),
$$
\sum_{i \geq 0} i\overline \mu_{2i+1}+ \sum_{i \geq 0} i\overline \mu_{2i+2}  =\frac{1}{4}\left(\sum_{k \geq 1} (k^2-2k) \mu_k+\mu_{2\mathbb N+1}\right)=\frac{\sigma^2-1+\mu_{2 \mathbb N+1}}{4}=\frac{\sigma^2-\mu_{2 \mathbb Z_+}}{4},
$$
which leads to the desired expression for
$
c_{\mathrm{\mathrm{geo}}}( \mu)
$.
\endproof

\bigskip

The strong law of large numbers applied to the Markov chain $(X_n,P_n)$ hence implies that $ n^{-1}S_{n} $ converges  to  $c_{\mathrm{\mathrm{geo}}}( \mu)$ almost surely as $n \to \infty$. 
For the proof of Theorem \ref {thm:main}, we will need an estimate of the speed of the latter convergence. To this end,  we establish the following large deviations result.

\begin{proposition}\label{prop:dev}For every $\epsilon>0$, there exist a constant $B( \epsilon)>0$ and an integer $n_{\epsilon}$ such that, for all $n \geq n_{\epsilon}$, 
\begin{equation}
\label{eq:dev}
\mathbb P \left( \left|\frac{S_n}{n} - c_{ \mathrm{\mathrm{geo}}}( \mu) \right| \geq \epsilon \right) \leq  \exp(-B( \epsilon) \cdot n).
\end{equation}
\end{proposition}

\begin{proof}
Recall that $\sum_{i\geq 0} e^{\lambda i}\mu_i <\infty$ for a certain $\lambda>0$.  When $\mu_{2 \mathbb N}=0$,  $(S_n)$ is a standard random walk (with i.i.d. increments), with step distribution having exponential moments.  The bound (\ref{eq:dev}) is then a standard large deviations result.
To prove a similar result when $\mu_{2 \mathbb N}>0$ (which we now assume), we use Theorem 5.1 of \cite{INN85}. 
According to this theorem, (\ref{eq:dev}) holds as soon as the following three conditions are satisfied:
\begin{enumerate}
\item the driving chain $(P_n)$ is irreducible aperiodic;
\item the chain $(X_n,P_n)$ satisfies the following recurrence condition: there exist $m_0 \geq 1$ and a non-zero measure $\nu$ on $\mathbb Z_+ \times \{\mathrm D, \mathrm U \}$ and constants $a,b \in (0,\infty)$ such that
\begin{equation}
\label{recurrence}
a\nu(i,\mathrm X) \leq \mathbb P\left(X_{n+m_0}=i, P_{n+m_0}=\mathrm X \ | \ P_n=\mathrm Y\right) \leq b \nu(i,\mathrm X)
\end{equation}
for every $i \in \mathbb Z_+$ and $\mathrm X, \mathrm Y \in \{\mathrm D, \mathrm U\}$. 
\item there exists $ \alpha>0$ such that \begin{equation}
\label{eq:sum}\sum_{i\geq 0}\exp(\alpha i) (\nu(i,\mathrm D)+\nu(i,\mathrm U)) <\infty.
\end{equation}
\end{enumerate}
To be completely accurate, Theorem 5.1 of \cite{INN85} actually assumes that the set of all $\alpha>0$ such that \eqref{eq:sum} holds  is open. However, by analyzing the proof, it turns out that this extra condition is only needed to get a lower large deviations bound.

By Lemma \ref{LemmaDriving}, we know that the driving chain is irreducible aperiodic when $\mu_{2 \mathbb N}>0$. To check the second condition, we will need the explicit expression of the two-step transition probabilities:
\begin{eqnarray*}
&&\mathbb P\left(X_{n+2}=i, P_{n+2}=\mathrm D \ | \ P_n=\mathrm D\right)=(\mu_{2 \mathbb N+1}+\mu_0)(\overline \mu_{2i+1} +\overline \mu_{2i+1}\mathbbm{1}_{\{i \geq 1\}} )+ \mu_{2 \mathbb N}2\overline \mu_{2i+2} \\
&&\mathbb P\left(X_{n+2}=i, P_{n+2}=\mathrm U \ | \ P_n=\mathrm D \right)= (\mu_{2 \mathbb N+1}+\mu_0) \mu_{2i}\mathbbm{1}_{\{i \geq 1\}} +\mu_{2 \mathbb N} \mu_{2i+1} \\
&&\mathbb P\left(X_{n+2}=i, P_{n+2}=\mathrm D \ | \ P_n=\mathrm U \right)=\mu_{2 \mathbb Z_+} (\overline \mu_{2i+1} +\overline \mu_{2i+1}\mathbbm{1}_{\{i \geq 1\}} )+ \mu_{2 \mathbb N+1} 2\overline \mu_{2i+2}\\
&& \mathbb P\left(X_{n+2}=i, P_{n+2}=\mathrm U \ | \ P_n=\mathrm U \right)=\mu_{2 \mathbb Z_+}\mu_{2i}\mathbbm{1}_{\{i \geq 1\}}+ \mu_{2 \mathbb N+1} \mu_{2i+1}.
\end{eqnarray*}
This suggests to set
\begin{equation*}
\nu(i,\mathrm D)=\overline \mu_{2i+1} +\overline \mu_{2i+1}\mathbbm{1}_{\{i \geq 1\}}+ 2\overline \mu_{2i+2} \quad \text{and} \quad \nu(i,\mathrm U)=\mu_{2i}\mathbbm{1}_{\{i \geq 1\}}+\mu_{2i+1}.
\end{equation*}
Assuming then that $\mu_{2\mathbb N+1}>0$, it is easy to check that (\ref{recurrence}) is satisfied with the two constants  $a=\min{(\mu_{2 \mathbb N},\mu_{2 \mathbb N+1})}$ and $b=1$
(and $m_0=2$).
Next, if $\mu_{2\mathbb N+1}=0$, notice that $\overline \mu_{2i+1}=\overline \mu_{2i+2}$ for all $i$, so that\begin{eqnarray*}
\nu(i,\mathrm D)=3\overline \mu_{2i+2} +\overline \mu_{2i+2}\mathbbm{1}_{\{i \geq 1\}} \quad \text{and} \quad
 \nu(i,\mathrm U)=\mu_{2i}\mathbbm{1}_{\{i \geq 1\}}.
\end{eqnarray*}
The inequalities (\ref{recurrence}) thus hold with the constants
$a= \mu_{2 \mathbb N}/3$ and $b=1$
(notice that $\mu_0 \geq 1/2 \geq \mu_{2 \mathbb N}/3$).
Hence, in all cases the second condition is satisfied. Finally, the last condition clearly holds since we have assumed that $\mu$ has exponential moments and since
$$\sum_{i\geq 1}\exp(\alpha i) (\nu(i,\mathrm D)+\nu(i,\mathrm U))=\sum_{i\geq 1}\exp(\alpha i) \left(2 \overline \mu_{2i+1}+2 \overline \mu_{2i+2} + \mu_{2i}+\mu_{2i+1}\right). \qedhere$$
\end{proof}

\section{Convergence towards the Brownian CRT} \label{sec:proof}

\subsection {The Gromov--Hausdorff topology}
\label{sec:GH}
 We start by recalling the definition of the Gromov--Hausdorff topology (see \cite{BBI01, Eva08} for additional details). If $(E,d)$ and $(E',d')$ are two  compact  metric spaces, the Gromov--Hausdorff distance between
${E}$ and ${E'}$  is defined by
 \begin{eqnarray*} \op{d_{GH}}({E},{E'}) &=& \inf \left\{\op{d}_{\op{H}}^F(\phi(E),\phi'(E'))\right\}, \end{eqnarray*} where the infimum is taken over all choices of metric spaces $(F,\delta)$ and  isometric embeddings $\phi : E \to F$ and $\phi'  : E' \to F$ of $E$ and $E'$ into $F$, and where $ \mathrm{d}_{ \mathrm{H}}^F$ is the Hausdorff distance between compacts sets in $F$. The  Gromov--Hausdorff distance is indeed a metric on the space of all isometry classes of compact metric spaces, which makes it separable and complete.
 
 An alternative practical definition of  $\op{d_{GH}}$ uses \textit{correspondences.} A correspondence between two  metric spaces $(E,d)$ and $(E',d')$ is by definition a subset $\mathcal{R} \subset E\times E'$ such that, for every $x_{1} \in E$, there exists at least one point $x_{2}\in E'$ such that $(x_{1},x_{2}) \in \mathcal{R}$ and conversely, for every $y_{2}\in E'$, there exists at least one point $y_{1}\in E$ such that $(y_{1},y_{2}) \in \mathcal{R}$. The distortion of the correspondence $\mathcal{R}$ is defined by 
 $$ \op{dis}(\mathcal{R}) = \sup\big\{|d(x_{1},y_{1})-d'(x_{2},y_{2})| : (x_{1},x_{2}),(y_{1},y_{2}) \in \mathcal{R} \big\}.$$ The Gromov--Hausdorff distance can then be expressed in terms of correspondences by the formula
 \begin{equation}
 \label{GHcorres}
 \op{d_{GH}}({E},{E'})= \frac{1}{2} \inf_{\mathcal{R} \subset E\times E'} \big\{\hspace{-0.5mm}\op{dis}(\mathcal{R})\big\},
 \end{equation} where the infimum is over all correspondences $\mathcal{R}$ between $(E,d)$ and $(E',d')$.

\subsection {Proof of Theorem \ref{thm:main}}
\label {sec:proof}

We first need to introduce some notation.   Let $ \tau \neq \{\varnothing\}$ be a finite tree such that no vertex has a unique child. Recall that $\tau^\bullet$ is the planted tree obtained from $ \tau$ by attaching an additional leaf  at the root and 
 denote  by $ \mathrm{d}_{\tau^\bullet}(u,v)$ the graph distance between $u, v \in \tau^\bullet$. From Section \ref{sec:geoexplo}, recall also that $ \ell_{0}, \ldots ,\ell_{\lambda(\tau)}$ are the leaves (in clockwise order) of $\tau^\bullet$. If $u,v$ are leaves of $\tau^\bullet$, let $0 \leq  p,q \leq \lambda( \tau)$ be such that $ u = \ell_{p}$ and $v= \ell_{q}$. Then denote by $\mathcal{D}= \phi^{-1}( \tau)$ the random dissection associated with $ \tau$ by duality (see Section \ref {sec:duality}). With a slight abuse of notation, we let $ \mathrm{d}_{ \mathcal{D}}(u,v)$ be the distance between $\overline{p}$ and $ \overline{q}$ in $ \mathcal{D}$.
 
 We say that a sequence of positive numbers $(x_n)_{n \geq 0}$ is $oe(n)$ if there exist  constants $ a, c, C >0$ such that $x_n \leq
C e^{-cn^{a}}$ for every $n \geq 0$, and we write
$x_n=oe(n)$.  Finally, fix $ \epsilon>0$ and  set $$ \epsilon_{n}( \tau^\bullet)=\epsilon  \max( \mathsf{Diam}( \tau^\bullet), \sqrt{n}).$$

\begin{lemma} \label{lem:controle}We have $$\GWmu{ \exists \, u,v \textrm{ leaves in } \tau^\bullet, \, \big|  \mathrm{d}_{ \mathcal{D}}(  u,  v)-c_{\mathrm{\mathrm{geo}}}( \mu)    \mathrm{d}_{\tau^\bullet}( u, v)\big| \geq  \epsilon_{n}( \tau^\bullet) \,\Big|\, \lambda( \tau^\bullet)=n}=oe(n).$$
\end{lemma}

\begin{proof}
Recall the notation $|u|$ for the generation of a vertex $u$ of a tree. We start by comparing the distance between a leaf and the root in the dissection and in the tree and show that
\begin{equation}
\label{eq:u1}
\GWmu{ \exists \, u  \textrm{ leaf in } \tau^\bullet,\, \big|\ddiss(  \varnothing,  u)-c_{\mathrm{\mathrm{geo}}}( \mu)  |u|\big | \geq  \epsilon_{n}( \tau^\bullet) \,\Big| \, \lambda( \tau^\bullet)=n}=oe(n).\end{equation}
For this, we use the notation $H_{\tau^\bullet}(u)$ introduced at the end of Section \ref {sec:geoexplo}. By \eqref{eq:procheun}, we have $|d_{ \mathcal{D}}( \varnothing, u)- H_{\tau^\bullet}(u)| \leq 1$ for every leaf $u \in \tau^\bullet$. In addition, by \cite[Theorem 3.1]{Kor12}, we have $${\GWmu{ \lambda( \tau^\bullet)=n}\quad =\quad \GWmu{ \lambda( \tau)=n-1}  \quad\mathop{ \sim}_{n \rightarrow \infty} \quad  \sqrt{ \frac{ \mu_{0}}{ 2 \pi \sigma^2}} \cdot \frac{1}{n^{3/2}},}$$
so  that $oe(n)/\GWmu{ \lambda( \tau^\bullet)=n}= oe(n)$.  Thus \eqref{eq:u1} will follow if we can show that
\begin{equation} \label{eq:u11}
\GWmu{ \exists \, u \in \tau^\bullet , \, \big |H_{\tau^\bullet}( {u})-c_{\mathrm{\mathrm{geo}}}( \mu)  |u| \big | \geq  \epsilon_{n}( \tau^\bullet)}=oe(n).
\end{equation}
To this end, we bound from above the left-hand side of \eqref{eq:u11} by 
\begin{eqnarray*}
&& \Esmu { \sum_{u \in \tau^\bullet}  \mathbbm{1}\left\{ \big| {H_{ \tau^\bullet}(u)}{} - c_{\mathrm{\mathrm{geo}}}( \mu) |u|\big| \geq   \epsilon_{n}( \tau^\bullet)\right\}}  \\
   &&   \qquad\qquad \qquad \qquad\qquad  \qquad = \sum_{j=1}^{ \infty} \Esmu { \sum_{u \in \tau^\bullet, |u| =j}  \mathbbm{1}\left\{ \left| \frac{H_{ \tau^\bullet}({u})}{j} - c_{\mathrm{\mathrm{geo}}}( \mu) \right|   \geq \frac{\epsilon_{n}( \tau^\bullet)}{j}\right\}} \\
   &&   \qquad\qquad \qquad \qquad\qquad  \qquad =\sum_{j=1}^{ \infty} \Esmu { \sum_{u \in \tau^\bullet, |u| =j}  \mathbbm{1}\left\{ \left| \frac{H_{ [\tau^\bullet]_{j}}({u})}{j} - c_{\mathrm{\mathrm{geo}}}( \mu) \right|   \geq \frac{\epsilon_{n}( \tau^\bullet)}{j}\right\}} \\
      &&   \qquad\qquad \qquad \qquad\qquad  \qquad  \leq \sum_{j=1}^{ \infty} \Esmu { \sum_{u \in \tau^\bullet, |u| =j}  \mathbbm{1}\left\{ \left| \frac{H_{ [\tau^\bullet]_{j}}({u})}{j} - c_{\mathrm{\mathrm{geo}}}( \mu) \right|   \geq \frac{ \epsilon \max(j,\sqrt{n})}{j}\right\}}.
\end{eqnarray*}
For the last inequality, we have used the fact that if there exists  $ u \in \tau^\bullet$ with $ |u|=j$ then $ \mathsf{Diam}( \tau^\bullet) \geq j$. Hence, using Proposition \ref{prop:biais} and then \eqref{eq:snh},  we get
  \begin{eqnarray*}
\GWmu{ \exists \, u \in \tau^\bullet ; \, \big |H_{ \tau^\bullet}(u)-c_{\mathrm{\mathrm{geo}}}( \mu)  |u| \big | \geq  \epsilon_{n}( \tau^\bullet)} &\leq&
 \sum_{j=1}^{ \infty}  {  \Pr{ \left| \frac{H_{[ T_{ \infty}^\bullet]_{j}}({W_{j}})}{j} - c_{\mathrm{\mathrm{geo}}}( \mu) \right|   \geq\frac{\epsilon \max( j, \sqrt{n})}{j}}}   \\
   & =  &  \sum_{j=1}^{ \infty} \Pr { \left|\frac{S_{j-1}}{j} - c_{\mathrm{\mathrm{geo}}}( \mu) \right| \geq    \frac{\epsilon \max( j, \sqrt{n})}{j}}.
\end {eqnarray*}
Now,  suppose that $n> n_{ \epsilon}^{4}$, so that Proposition \ref{prop:dev} can be applied:
   \begin{eqnarray*}
   \sum_{j=n^{1/4}}^{ \infty} \Pr { \left|\frac{S_{j-1}}{j} - c_{\mathrm{\mathrm{geo}}}( \mu) \right| \geq    \frac{\epsilon \max( j, \sqrt{n})}{j}} & \leq & \sum_{j=n^{1/4}}^{ \infty} \Pr { \left|\frac{S_{j-1}}{j} - c_{\mathrm{\mathrm{geo}}}( \mu) \right| \geq   \epsilon} =oe(n).
      \end{eqnarray*}
Assume in addition that $n$ is sufficiently large so that $  \epsilon n^{1/4}> c_{ \textrm{geo}}( \mu)$. In order to bound the remaining terms corresponding to $1 \leq j \leq  n^ {1/4}$, note that if    $|{S_{j-1}}/{j} - c_{\mathrm{\mathrm{geo}}}( \mu) | \geq  \epsilon n^{1/4}$ for some $1 \leq  j \leq n^{1/4}$, then necessarily there exists $0 \leq i \leq n^{1/4}$ such that $S_{i+1}-S_{i}> \epsilon n^{1/4}$.  Then note from Section \ref{sec:makov}, with the notation introduced there, that $S_{i+1}-S_{i} \leq 1 + \max (G_{i},D_{i})$. Since the variables $(G_{i},D_{i})_{i \geq 1}$ are i.i.d.\,\,with exponential moments, by combining an exponential Markov inequality with a union bound      
     we easily get that for every $j \leq n^{1/4}$, $\Pr { |{S_{j-1}}/{j} - c_{\mathrm{\mathrm{geo}}}( \mu) | \geq    {\epsilon \max( j, \sqrt{n})}/{j}}=oe(n)$. Therefore
      $$
   \sum_{j=1}^{n^{1/4}} \Pr { \left|\frac{S_{j-1}}{j} - c_{\mathrm{\mathrm{geo}}}( \mu) \right| \geq    \frac{\epsilon \max( j, \sqrt{n})}{j}}  =oe(n),$$  which establishes \eqref{eq:u11} and hence \eqref{eq:u1}.

To conclude, we use the rotational invariance of Boltzmann dissections. 
Conditionally on $ \tau^\bullet$, let $ \mathcal{L}$ and $ \mathcal{L'}$ be two leaves chosen independently and uniformly at random from $  \tau^\bullet$. 
Then, under  $ \GW_{ \mu}$,
\begin{equation}
\label{eq:reroot}( \ddiss( \mathcal{L}, \mathcal{L'}),  \mathrm{d}_{\tau^\bullet}( \mathcal{L}, \mathcal{L'}),  \tau^{\bullet,[\mathcal L]}) \quad  \quad\mathop{=}^{(d)} \quad (\ddiss(  \varnothing  , \mathcal{L}), \mathrm{d}_{\tau^\bullet}( \varnothing, \mathcal{L}), \tau^\bullet),
\end{equation}
where $ \tau^{\bullet,[\mathcal L]}$ denotes the planted tree $\tau^{\bullet}$ re-rooted at $\mathcal L$.
Note that it is crucial here to work with the planted version of trees.
Hence, by \eqref{eq:u1}, we get that
\begin{equation}
\label{eq:uu}\GWmu{\big|  \ddiss( \mathcal{L}, \mathcal{L'})-c_{\mathrm{\mathrm{geo}}}( \mu)   \mathrm{d}_{\tau^\bullet}( \mathcal{L}, \mathcal{L'})\big| \geq \epsilon_{n}( \tau^\bullet) \,\Big|\, \lambda( \tau^\bullet)=n}  =oe(n).
\end{equation}
Now, conditionally on $ \tau^\bullet$,  let $( \mathcal{L}_{j}, \mathcal{L}'_{j})_{ 1 \leq j \leq   \lambda(\tau^\bullet)^3}$ be a sequence of i.i.d. couples of  independent uniform leaves. Conditionally on $\lambda( \tau^\bullet)=n$, the probability that there exists $1 \leq j \leq n^{3}$  such that  $|  \ddiss( \mathcal{L}_{j}, \mathcal{L}'_{j})-c_{\mathrm{\mathrm{geo}}}( \mu)   \mathrm{d}_{\tau^\bullet}( \mathcal{L}_{j}, \mathcal{L'}_{j})| \geq \epsilon_{n}( \tau^\bullet)$ is smaller than $n^3 oe(n)=oe(n)$ by \eqref{eq:uu}. On the other hand, conditionally on $ \lambda(\tau^\bullet)=n$, 
the probability that there exists a couple of leaves of $\tau^{\bullet}$ which does not belong to $( \mathcal{L}_{j}, \mathcal{L}'_{j})_{ 1 \leq j \leq   n^3}$ is smaller than $ n^2 \left( 1-{n^  {-2}}\right)^{n^3}=oe(n)$. Hence
$$\GWmu{ \exists \, u,v \textrm{ leaves in } \tau^\bullet, \, \big|  \mathrm{d}_{ \mathcal{D}}(  u,  v)-c_{\mathrm{\mathrm{geo}}}( \mu)    \mathrm{d}_{\tau^\bullet}( u, v)\big| \geq  \epsilon_{n}( \tau^\bullet) \,\Big|\, \lambda( \tau^\bullet)=n} \leq oe(n)+oe(n). \qedhere$$
\end{proof}

The next proposition will lead to an effortless proof of Theorem \ref{thm:main}, as well as  interesting applications to the convergence of moments. If $ \tau$ is a finite tree, we denote by $ \tau^{ \ell}$  the graph formed by the leaves of $ \tau$ equipped with the graph distance of $\tau$. Recall that $  \mathcal{D}^\mu_{n} $ denotes
a random dissection of $ \mathcal{P}_{n}$ distributed according to $ \mathbb{P} ^ {\mu}_ {n}$ and let $ \mathcal{T}_{n}= \phi(  \mathcal{D}^\mu_{n}) ^{ \bullet}$ be its dual planted tree. By Proposition \ref {prop:GW}, $ \mathcal{T}_{n}$ has the same distribution as  the planted version of a $  \GW_{ \mu}$ tree conditioned on having $n-1$ leaves. Finally, recall the notation $ \epsilon_{n}( \cdot)$ introduced just before Lemma \ref{lem:controle}.

\begin{proposition}We have:
\label{prop:utile}
\begin{enumerate}
 \item[(i)] $ \displaystyle \mathrm{d_{GH}} \left( { \mathcal{T}^{ \ell}_{n}}{ },  { \mathcal{T}_{n}}{ }\right) \leq  \frac{ \ln(n)}{ \ln(2)}$,
 \item[(ii)] $ \displaystyle \Pr{ \mathrm{d_{GH}} \left(\mathcal{D}^\mu_{n},c_{\mathrm{geo}}( \mu) \cdot { \mathcal{T}^{ \ell}_{n}} \right) \geq  \epsilon_{n}( \mathcal{T}_{n})}  =oe(n)$.
  \end{enumerate}
\end{proposition}

\begin{proof}
The first assertion comes from the following deterministic observation: if $  \tau_{n}$ is a tree with $n$ vertices such that no vertex has a unique child, then
\begin{equation}
\label{eq:borne} \mathrm{ d_{GH}} \left( { \tau^{ \ell}_{n}}{ },  {\tau_{n}}{ }\right) \leq  \frac{ \ln(n)}{ \ln(2)}.
\end{equation}
Indeed, for $u \in \tau_{n}$, denote by $u^ \ell$ a leaf with lowest generation among the descendants of $u$. If $ | u^ \ell| \geq k+|u|$, then there are at least $2^k$ vertices among the $k$-th generation descending from $u$. Hence $ 2^k \leq n$, so that $ k \leq  \ln(n)/ \ln(2)$. As a consequence, every vertex of $ \tau_{n}$ has a leaf at distance at most $ \ln(n)/ \ln(2)$ and \eqref{eq:borne} follows.

The second assertion is an immediate consequence of Lemma \ref{lem:controle} by considering the trivial correspondence $\{(\overline k, \ell_{ k}), 0 \leq k \leq n-1\} \subset \mathcal{D}^\mu_{n} \times  (c_{\mathrm{geo}}(\mu) \cdot \mathcal{T}^{ \ell}_{n})$. 
\end{proof}

\begin{proof}[Proof of Theorem \ref{thm:main}]
By \cite{Riz11,Kor12}, $\mathcal{T}_{n}/ \sqrt{n}$ converges in distribution for the Gromov--Hausdorff topology towards $c_{ \mathrm{tree}}(\mu)   \cdot  \mathcal{T}_{ \mathbf{e}}$ as $n \rightarrow \infty$. Hence, by Proposition \ref{prop:utile} (i), $\mathcal{T}^{ \ell}_{n}/ \sqrt{n}$ converges in distribution towards $c_{ \mathrm{tree}}(\mu)   \cdot  \mathcal{T}_{ \mathbf{e}}$. It is thus sufficient to establish that
\begin{equation}
\label{eq:p1} \mathrm{d_{GH}} \left(\frac{ \mathcal{D}^\mu_{n}}{ \sqrt{n}},c_{\mathrm{\mathrm{geo}}}( \mu)\frac{ \mathcal{T}^{ \ell}_{n}}{ \sqrt{n}} \right)    \qquad\mathop{\longrightarrow}_{n \rightarrow \infty}^{ ( \P)} \qquad 0.
\end{equation}
From Proposition \ref{prop:utile} (ii),
$$ {\frac{\sqrt{n}}{ \max(\sqrt{n}, \mathsf{Diam}( \mathcal{T}_n))}  \  \mathrm{d_{GH}} \left(\frac{ \mathcal{D}^\mu_{n}}{ \sqrt{n}},c_{\mathrm{\mathrm{geo}}}( \mu)\frac{ \mathcal{T}^{ \ell}_{n}}{ \sqrt{n}} \right) }\quad\mathop{\longrightarrow}_{n \rightarrow \infty}^{( \P)} \quad 0.$$
Moreover, since $\mathcal{T}_{n}/ \sqrt{n}$ converges in distribution towards $c_{ \mathrm{tree}}(\mu)   \cdot  \mathcal{T}_{ \mathbf{e}}$,  the random variable \linebreak $ \max(\sqrt{n}, \mathsf{Diam}( \mathcal{T}_n))/\sqrt{n} $ converges in distribution  towards an a.s. finite random variable. Convergence \eqref{eq:p1} hence follows, and this completes the proof.\end{proof}

\section{Applications}
\label{sec:appl}

\subsection{Convergence of moments for different statistics}
\label{sec:moments}

The following result strengthens Theorem \ref{thm:main} and  will lead to asymptotic estimates for moments of various statistics of $\mathcal{D}^\mu_{n}$.

\begin{proposition}\label{prop:moments} Let $F$ be a positive continuous  function defined on the set of all (isometry classes of) compact metric spaces, such that$F( \mathcal{M})\leq C \mathsf{Diam}(\mathcal{M})^p$ for all compact metric spaces $\mathcal M$ and fixed $C,p > 0$. Then:
$$ \Es{F \left( \frac{\mathcal{D}^\mu_{n}}{\sqrt{n}}\right)}  \quad\mathop{\longrightarrow}_{n \rightarrow \infty} \quad  \Es{F(c(\mu)\cdot \mathcal{T}_{ \mathbf{e}})}.$$
\end{proposition}
 Let $ \H( \tau)$ denote the height of a finite tree $ \tau$. The main tool to prove Proposition \ref{prop:moments} is the following bound on the height of large conditioned Galton--Watson trees, which is a particular case of \cite[Lemma 33]{HM12}.

\begin{lemma} \label{lem:height} For every $q >0$, there exists a constant $C_q<\infty$ such that, for every $n \geq 1$ and $s>0$,
{$$\mathsf{GW}_{\mu}\left( \H( \tau) \geq s n^{1/2} \, | \, \lambda( \tau)=n \right) \leq \frac{C_{q}}{s^q}.$$}
\end{lemma}

\begin{proof}[Proof of Proposition \ref{prop:moments}]  By Theorem \ref{thm:main}, $F \left( {\mathcal{D}^\mu_{n}}/{\sqrt{n}}\right)$ converges in distribution towards $F(c(\mu)\cdot\mathcal{T}_{ \mathbf{e}})$. It is thus sufficient to check that $ \mathbb E\big[F \left( {\mathcal{D}^\mu_{n}}/{\sqrt{n}}\right)^2\big]$ is bounded as $ n \rightarrow \infty$. In the following lines, $C$ is a finite constant that may vary from line to line and may depend on $p$. {Since} $F( \mathcal{M}) \leq C \ \mathsf{Diam}( \mathcal{M})^p$  and since $ \mathsf{Diam}( \mathcal{M}) \leq  \mathsf{Diam}( \mathcal{N})+ 2 \mathrm{d_{GH}}( \mathcal{M}, \mathcal{N})$ for any two compact metric spaces $ \mathcal{M}$ and $ \mathcal{N}$ (see e.g. \cite[Exercise 7.3.14.]{BBI01}), the expectation $ \mathbb E \big[F \left( {\mathcal{D}^\mu_{n}}/{\sqrt{n}}\right)^2\big]$ is bounded above by the expression
$$  C \ \Es{\mathsf{Diam}\left(  \frac{ \mathcal{T}_{n}}{ \sqrt{n}} \right)^{2p}}+ C \ \Es{ \mathrm{d_{GH}} \left( \frac{ \mathcal{T}^{ \ell}_{n}}{\sqrt n}, \frac { \mathcal{T}_{n}}{\sqrt n}\right)^{2p}}+ C \ \Es{ \mathrm{d_{GH}} \left(\frac{ \mathcal{D}^\mu_{n}}{ \sqrt{n}},c_{\mathrm{\mathrm{geo}}}( \mu) \cdot \frac{ \mathcal{T}^{ \ell}_{n}}{ \sqrt{n}} \right) ^{2p}}.$$
By Lemma \ref{lem:height}, the first term of the last expression is bounded as $ n \rightarrow \infty$, and by Proposition \ref{prop:utile} (i), the second term is bounded as well. For the third term, first note that since the graphs ${ \mathcal{D}^\mu_{n}}$ and $\mathcal{T}^{ \ell}_{n}$ have $n$ vertices, they are at Gromov--Hausdorff distance at most $n$ from each other. Hence, using Proposition \ref{prop:utile} (ii), we get
\begin{eqnarray*}
\Es{ \mathrm{d_{GH}} \left(\frac{ \mathcal{D}^\mu_{n}}{ \sqrt{n}},c_{\mathrm{\mathrm{geo}}}( \mu) \cdot \frac{ \mathcal{T}^{ \ell}_{n}}{ \sqrt{n}} \right) ^{2p}} \leq n^{2p} \cdot  oe(n)+ C \ \Es{\max\left( \frac{ \mathsf{Diam}(\mathcal{T}_{n})} {\sqrt{n}},1\right)^{2p}} \leq C.
\end{eqnarray*}
This completes the proof.
\end{proof}

\subsubsection{Applications to the diameter}
\label{sec:diam}

Since $ \Es{\mathsf{Diam}( \mathcal{T}_{ \mathbf{e}})}= {2 \sqrt{2 \pi}}/{3}$ (see e.g. \cite[Section 3]{Ald91}), we get from Proposition \ref{prop:moments} that
$$\mathbb{E}\Big[\mathsf{Diam}( \mathcal{D}_{n}^\mu)\Big]   \displaystyle \quad\mathop{ \sim}_{n \rightarrow \infty} \quad  c( \mu) \ \frac{2 \sqrt {2\pi}}{3}  \ \sqrt{n}.$$
This gives a precise asymptotic estimate of the expected value of $\mathsf{Diam}(\mathcal D_n^{\mu})$ and improves results of \cite[Section 5]{DMN12} where bounds for the expected value of the diameter of uniform dissections and triangulations were found using a generating functions approach. More generally, for every $p > 0$,
$$\mathbb{E}\Big[\mathsf{Diam}( \mathcal{D}_{n}^\mu)^p\Big]   \displaystyle \quad\mathop{ \sim}_{n \rightarrow \infty} \quad  c( \mu)^p \int_{0}^ \infty x^pf_{D} (x) \mathrm dx  \cdot {n}^{p/2},$$ where
$f_D$, the density of the diameter of the Brownian tree, is given by $$ f_{D} (x)=\frac{2\sqrt{2 \pi}}{3}\sum_{k \geq 1}\left(\frac{4}{x^4} \left(4  b_{k,x}^4-36  b_{k,x}^3+ 75  b_{k,x}^2 -30  b_{k,x} \right) +\frac{2}{x^2}\left( 4  b_{k,x}^3-10  b_{k,x}^2\right) \right)\exp(- b_{k,x}), $$ with  $b_{k,x}=\left( 4\pi k/x \right)^2$ for  $x>0$ (see e.g. \cite{Sze83} and \cite[Section 3]{Ald91}).

\subsubsection{Applications to the radius}

Let $ \mathsf{Radius}(\mathcal{D}^\mu_{n})$ denote the maximal distance of a vertex of $ \mathcal{D}^\mu_{n}$ to the vertex $\overline 0$. A simple extension of Theorem \ref{thm:main} and Proposition \ref{prop:moments} to the pointed Gromov--Hausdorff topology (see e.g. \cite{LG05}), entails that, for every $ p > 0$,  $ \Es{ \mathsf{Radius}(\mathcal{D}^\mu_{n})^p}$ is asymptotic to $c( \mu)^p  \Es{\mathsf{Height}(\mathcal{T}_{ \mathbf{e}})^p} n^{p/2}$ as $ n \rightarrow \infty$. Using the explicit expression for $ \Es{\mathsf{Height}(\mathcal{T}_{ \mathbf{e}})^p}$ in \cite{BPY01}, we get
$$  \Es{ \mathsf{Radius}(\mathcal{D}^\mu_{n})^p}  \quad\mathop{ \sim}_{n \rightarrow \infty} \quad  c( \mu)^p \ 2^{-p/2}p(p-1) \Gamma (p/2)\zeta(p) \ n^{p/2},$$
where $\Gamma$ denotes Euler's gamma function and $\zeta$ Riemann's zeta function. In particular, for $p=1$, we get
$$  \Es{ \mathsf{Radius}(\mathcal{D}^\mu_{n})}  \quad\mathop{ \sim}_{n \rightarrow \infty} \quad  c( \mu)  \  \sqrt{\frac{{\pi}}{2}} \  \sqrt{n}.$$
In \cite{DMN12}, this result has been established for uniform dissections and uniform triangulations by using a generating functions approach.

\subsubsection{Applications to the height of a uniform leaf}

 Let $ \mathsf{Height}_{U}(\mathcal{D}^\mu_{n})$ denote the distance to the vertex $ \overline{0}$ of a vertex of $ \mathcal{D}^\mu_{n}$ chosen uniformly at random.  A simple extension of Theorem \ref{thm:main} and Proposition \ref{prop:moments} to the two-pointed Gromov--Hausdorff topology, entails that, for every $ p > 0$, $ \Es{ \mathsf{Height}_{U}(\mathcal{D}^\mu_{n})^p}$ is asymptotic to $c( \mu)^p  \Es{\mathsf{Height}_{U}(\mathcal{T}_{ \mathbf{e}})^p} n^{p/2}$ as $ n \rightarrow \infty$, where $\mathsf{Height}_{U}(\mathcal{T}_{ \mathbf{e}})$ is the height of a uniformly chosen point  of $ \mathcal{T}_{ \mathbf{e}}$. Since the random variable ${\mathsf{Height}_U(\mathcal{T}_{ \mathbf{e}})}$ has density $4x\exp(-2x^2)$ (see e.g. \cite[Section 3]{Ald91}), we get
$$  \Es{  \mathsf{Height}_{U}(\mathcal{D}^\mu_{n})^p}  \quad\mathop{ \sim}_{n \rightarrow \infty} \quad  c( \mu)^p \ 2^{-p/2}  \Gamma (1+p/2) \ n^{p/2}.$$
In particular, for $p=1$, we get
$ \   \displaystyle  \Es{  \mathsf{Height}_{U}(\mathcal{D}^\mu_{n})}  \mathop{ \sim}_{n \rightarrow \infty}   c( \mu)  \ \frac{1}{2} \sqrt{\frac{{\pi}}{2}} \  \sqrt{n}.$

\subsection{Examples: Dissection with constrained face degrees}
\label{examples}

Let $\mathcal{A}$ be a non-empty subset of $\{3,4,5,\ldots\}$ and let $\mathbf{D}^{(\mathcal{A})}_n$ be the set of all dissections of $ \mathcal{P}_{n}$ whose face degrees all belong to the set $\mathcal{A}$. We restrict our attention to the values of $n$ for which $ \mathbf{D}_{n}^{( \mathcal{A})} \ne  \varnothing$.  Let
$\mathcal{D}^{(\mathcal{A})}_n$ be  uniformly distributed  over
$\mathbf{D}^{(\mathcal{A})}_n$. By \cite[Section 3.1.1]{CKdissections},  $\mathcal{D}^{(\mathcal{A})}_n$  is distributed according to the Boltzmann probability measure $\mathbb{P} ^ {\nu_{\mathcal{A}}}_ {n}$   for a certain probability measure $\nu_{\mathcal{A}}$ defined as follows. Denote by $\mathcal{A}-1$ the set $ \{ a-1 : \, a \in \mathcal{A}\}$ and  let $r_{\mathcal{A}} \in
(0,1)$ be the unique real number in $(0,1)$ such that
$$\sum_{i \in \mathcal{A}-1} i r_{\mathcal{A}}^{i-1}=1.$$
Then $\nu_{\mathcal{A}}$ is defined by
$$\nu_{\mathcal{A}}(0)=1-\sum_{i \in \mathcal{A}-1} r_{\mathcal{A}}^{i-1}, \qquad\nu_{\mathcal{A}}(i)=r_{\mathcal{A}}^{i-1}  \textrm{  for }i \in
\mathcal{A}-1.$$ Note that the assumptions of Theorem \ref{thm:main} are satisfied. Hence, setting $c_{ \mathcal{A}}= c(\nu_{\mathcal{A}})$ to simplify notation, we get:
$$ \frac{1}{ \sqrt{n}} \cdot \mathcal{D}^ { ( \mathcal{A})}_{n}  \quad \xrightarrow[n\to\infty]{(d)} \quad      c_{ \mathcal{A}} \cdot  \mathcal{T}_{ \mathbf{e}},$$
together with the convergences of all positive moments of the different statistics mentioned in the previous section.

For uniform dissections ($A= \{3,4,5,\ldots\}$) and $p$-angulations for $p\geq 3$ ($A = \{p\}$), the scaling constants $c_{ \mathcal{A}}$ have been given in the Introduction. Let us mention two other interesting cases where $c_{ \mathcal{A}}$ is explicit (we leave the calculations to the reader): 
\begin{itemize}
\item  Only even face degrees ($A =   \{4,6,8, \ldots \}$). In this case $ c_A= \displaystyle \sqrt { \frac{1}{2}+ \frac{9}{2 \sqrt {17}}}  \simeq 1.2615.$
\item Only odd face degrees ($A =   \{3,5,7, \ldots \}$). In this case, the explicit expression of $ c_{ \mathcal{A}}$ is complicated (but available) and we only give a numerical approximation: $ c_{ \mathcal{A}} \simeq 1.0547$.

\end{itemize}

 \subsection{Extensions and discrete looptrees}

Let us mention some possible extensions of Theorem \ref{thm:main}. In one direction, it is natural to expect that Theorem \ref{thm:main} is still valid under the weaker assumption that $ \mu$ is critical and has finite variance. However our proof based on large deviation estimates seems unadapted and finer arguments would be needed.
In another direction, it would be interesting to extend Theorem \ref{thm:main} to other classes of so-called sub-critical graphs which also exhibit a tree-like structure, see \cite{DMN12,FN99}. \medskip

We now study the scaling limits of discrete looptrees associated with large conditioned Galton--Watson trees, which is a model similar  to the one of Boltzmann dissections: With every rooted oriented tree (or plane tree) $ \tau$, we associate a graph denoted by $ \mathsf{Loop}( \tau)$ and constructed by replacing each vertex $u \in \tau$ by a discrete cycle of length given by the degree of $u$ in $ \tau$ (i.e. number of neighbors of $u$) and gluing all these cycles according to the tree structure provided by $\tau$, see Figure \ref{fig:loop}. We view $\mathsf{Loop}( \tau)$ as a compact metric space by endowing its vertices with the graph distance.  
 \begin{figure}[!h]
 \begin{center}
 \includegraphics[width=0.6  \linewidth]{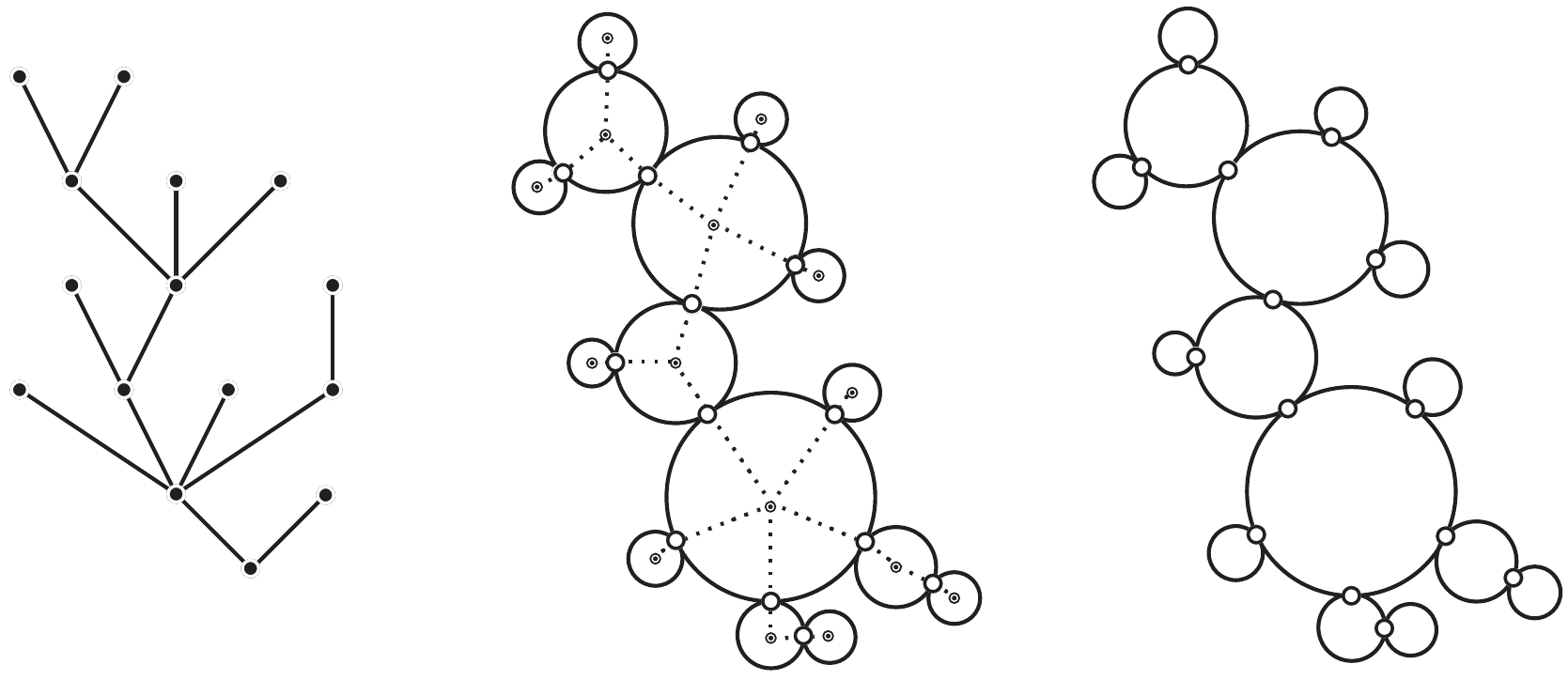}
 \caption{ \label{fig:loop}A discrete tree $\tau$ and its associated discrete looptree $ \mathsf{Loop}( \tau)$.}
 \end{center}
 \end{figure}
  Recall the notation $\mu_{0}+ \mu_{2} + \mu_{4} + \cdots = \mu_{2 \mathbb{Z}_{+}}$.
 
 \begin{theorem} \label{thm:loops} Let $\mu$ be a probability measure on $  \Z_{+}$ of mean $1$ and such that $ \sum_{k \geq 0}\mu_{k} e ^{\lambda k} < \infty$ for some $ \lambda >0$. For $ n \geq 1$, let $ \t_{n}$ be a $ \GW_{ \mu}$ tree conditioned on having $n$ vertices. Then we have the following convergence in distribution for the Gromov--Hausdorff topology
  \begin{eqnarray*} n^{-1/2} \cdot \mathsf{Loop}( \t_{n}) & \xrightarrow[n\to\infty]{(d)} &    { \frac{2}{ \sigma}} \cdot   \frac{1}{4}\left( \sigma^2+ 4- \mu_{2 \Z+} \right) \cdot  \mathcal{T}_{ \mathbf{e}}.  \end{eqnarray*}
  \end {theorem}
 
 The proof of Theorem \ref{thm:loops} goes along the same lines as that of Theorem \ref{thm:main}, but is much easier since here the Markov chain is just a random walk. We leave details to the reader. In \cite {CK13}, it is shown that when $\mu$ is a critical probability measure on $  \Z_{+}$ belonging to the domain of attraction of a stable law of index $ \alpha \in (1,2)$,  the random metric spaces $\mathsf{Loop}( \t_{n})$, appropriatly rescaled, converge towards the so-called random stable looptree of index $ \alpha$. Hence Theorem \ref{thm:loops} completes   \cite{CK13} by including the case where $ \mu$ has finite variance.
 \bigskip

 We end this paper by considering a model which is similar to the one of discrete looptres: With every rooted oriented tree (or plane tree) $ \tau$, we associate a graph denoted by $ \overline{\mathsf{Loop}}( \tau)$ constructed as follows. First consider the graph on the set of vertices of $\tau$ such that two vertices $u$ and $v$ are joined by an edge if and only if one of the following three conditions are satisfied in $ \tau$: $u$ and $v$ are consecutive siblings of a same parent, or $u$ is the first sibling (in the lexicographical order) of $v$, or $u$ is the last sibling of $v$. Then $ \overline{\mathsf{Loop}}( \tau)$ is by the definition the graph obtained  by contracting the edges $(u,v)$ such that $v$ is the last child of $ u$ in lexicographical order in $\tau$, see Figure \ref{fig:loop}. We view $ \overline{\mathsf{Loop}}( \tau)$ as a compact metric space by endowing its vertices with the graph distance.  
 \begin{figure}[!h]
 \begin{center}
 \includegraphics[width=0.8  \linewidth]{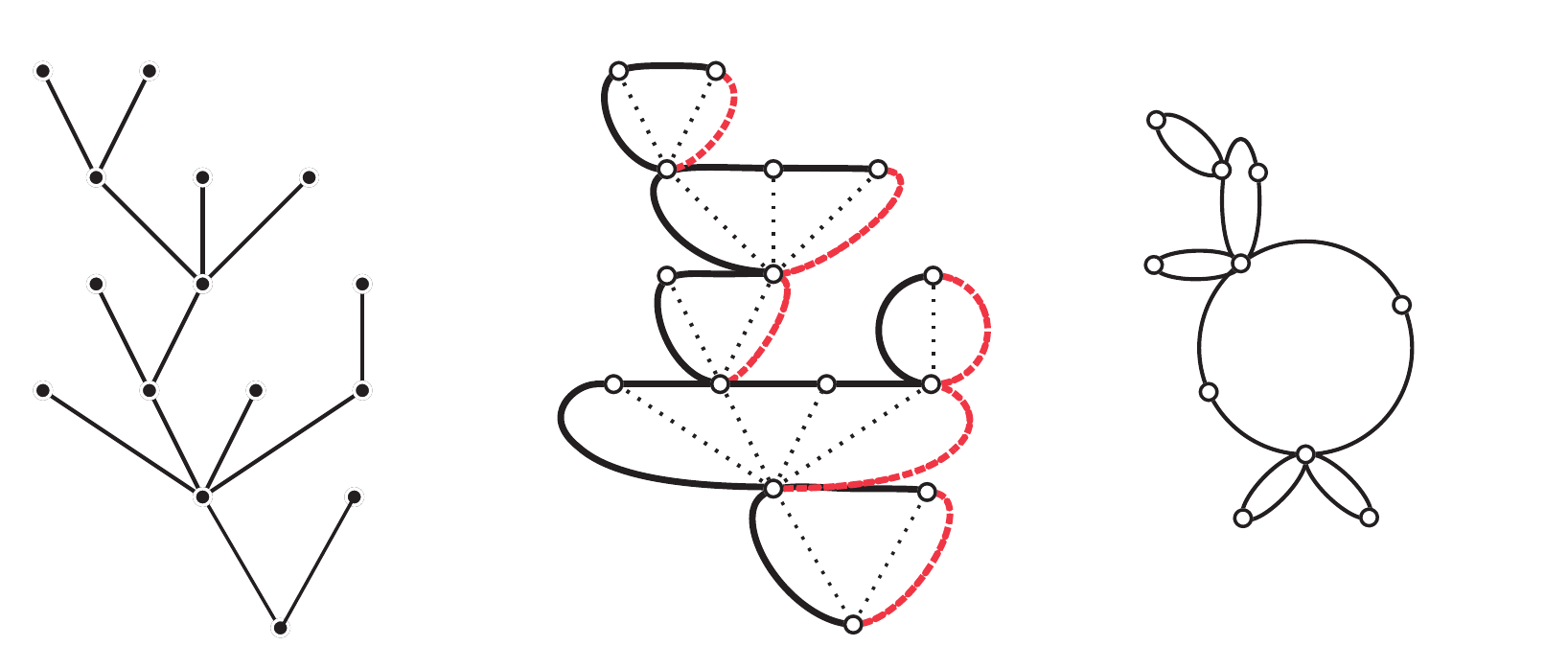}
 \caption{ \label{fig:loop}A discrete tree $\tau$ and its associated graph $  \overline{\mathsf{Loop}}( \tau)$. The contracted edges are bold, dashed and in red.}
 \end{center}
 \end{figure}
 
 \begin{theorem} \label{thm:loops2} Let $\mu$ be a probability measure on $  \Z_{+}$ of mean $1$ and such that $ \sum_{k \geq 0}\mu_{k} e ^{\lambda k} < \infty$ for some $ \lambda >0$. For $ n \geq 1$, let $ \t_{n}$ be a $ \GW_{ \mu}$ tree conditioned on having $n$ vertices. Then we have the following convergence in distribution for the Gromov--Hausdorff topology
  \begin{eqnarray*} n^{-1/2} \cdot \overline{\mathsf{Loop}}( \t_{n}) & \xrightarrow[n\to\infty]{(d)} &    { \frac{2}{ \sigma}} \cdot   \frac{1}{4}\left( \sigma^2+ \mu_{2 \Z+} \right) \cdot  \mathcal{T}_{ \mathbf{e}}.  \end{eqnarray*}
  \end {theorem}
 
 As for Theorem \ref {thm:loops}, the proof of Theorem \ref{thm:loops2} goes along the same lines as that of Theorem \ref{thm:main}, and is again easier since here the Markov chain is also just a random walk. We leave details to the reader. The motivation of this model comes from the fact that Theorem \ref{thm:loops2} has an interesting application to the study of the asymptotic behavior of subcritical site-percolation on large random triangulations \cite{CKpercolooptrees}. 

\bibliographystyle{siam}

\end{document}